\documentclass[11pt]{article}

\usepackage[english]{babel}
\usepackage[T1]{fontenc}
\usepackage[utf8]{inputenc}

\usepackage{times}
\usepackage{amsthm}
\usepackage{amsmath}
\usepackage{amssymb}
\usepackage{mathrsfs}
\usepackage{pb-diagram}
\usepackage{color}

\usepackage{hyperref}

\renewcommand\eqref[1]{(\ref{#1})}

\def\A{{\mathcal A}}

\def\O{\mathcal O}

\def\th{\theta}
\def\Th{\Theta}
\def\N{\mathbb{N}}
\def\T{\mathbb{T}}
\def\Z{\mathbb{Z}}

\def\D{{\mathcal D}}

\def\H{\mathcal H}

\def\R{\mathbb{R}}

\def\S{\mathcal S}

\def\e{{\sf e}}
\def\g{{\mathfrak g}}
\def\m{{\sf m}}
\def\wm{\widehat{\sf m}}

\def\({\left(}
\def\[{\left[}
\def\){\right)}
\def\]{\right]}

\def\si{\sigma}
\def\Si{\Sigma}
\def\bu{\bullet}

\def\G{{\sf G}}
\def\wG{\widehat{\sf{G}}}
\def\p{\parallel}
\def\<{\langle}
\def\>{\rangle}

\def\fscr{\mathscr}

\newtheorem{Theorem}{Theorem}[section]
\newtheorem{Remark}[Theorem]{Remark}

\newtheorem{Corollary}[Theorem]{Corollary}
\newtheorem{Proposition}[Theorem]{Proposition}
\newtheorem{Definition}[Theorem]{Definition}
\newtheorem{Example}[Theorem]{Example}

\numberwithin{equation}{section}

\begin{document}

\title{Global and Concrete Quantizations on General Type I Groups}

\date{\today}

\author{M. M\u antoiu and M. Sandoval 
}
\date{\small}
\maketitle \vspace{-1cm}

\begin{abstract}
In recent papers and books, a global quantization has been developed for \textit{unimodular} groups of type I\,. It involves operator-valued symbols defined on the product between the group $\mathsf{G}$ and its unitary dual $\widehat{\mathsf{G}}$\,, composed of equivalence classes of irreducible representations. For compact or for graded Lie groups, this has already been developed into a powerful pseudo-differential calculus. In the present article we extend the formalism to arbitrary locally compact groups of type I\,, making use of the Fourier theory of non-unimodular second countable groups. The unitary dual and its Plancherel measure being quite abstract in general, we put into evidence situations in which concrete forms are available. Kirillov theory and parametrizations of large parts of $\widehat{\mathsf{G}}$ allow rewriting the basic formulae in a manageable form. Some examples of completely solvable groups are worked out.
\end{abstract}

\section{Introduction}\label{humpf}

Aiming at a pseudo-differential theory working globally for classes of
non-commutative Lie groups, M. Ruzhansky and collaborators, following an
undeveloped idea of M. Taylor~\cite[Sect.\,1.2]{Ta} (see also~\cite{Ze}),
introduced quantizations making a full use of the group structure and the
harmonic analysis concepts connected with its representation theory. Namely,
denoting by $\wG$ the space of classes of equivalence of unitary irreducible
representations of $\G$\,, to suitable families
$\big\{a (x, \xi)\mid x \in\G\,,\xi\in\wG\big\}$\,, where $a (x,\xi)$ is an operator in the Hilbert space $\H_\xi$ of the representation $\xi$\,, one associates operators ${\sf Op}(a)$\,, acting in various function spaces over $\G$\,. Using specific properties, first the cases of compact Lie groups~\cite{RT} and graded nilpotent Lie groups~\cite{FR} were treated and developed into a very detailed formalism. H\"ormander-type of symbol classes are available and this has far-reaching consequences, many traditional results valid for $\G=\R^n$ being adapted to theses cases. In the books~\cite{RT,FR} many other recent relevant references may be found, among which we cite a few:~\cite{BFKG,CGGP,FR1,Glo,Melin,RT2,RTW,RW}. One can also find there plenty of results and applications, as well as a description of the advantages of such a framework, that we  do not repeat here.

\smallskip
The core of the theory can be extended to much more general classes of topological groups. In~\cite{MR}, one treats {\it unimodular} second countable type I locally compact groups (clearly in such a general case not all the techniques from~\cite{RT,FR} can be adapted). See also~\cite{Ma1} for the corresponding Berezin-Toeplitz formalism. However, {\it non-unimodular} type I groups form a rich and important class. Many simple and natural semi-direct products are of this type. Even in low dimension, the non-unimodular Lie groups prevail. Many other examples arise in the study of parabolic subgroups of semisimple Lie groups, that are used to investigate irreducible representations using extensions of Mackey's machine.

\smallskip
The first main purpose of this article is to make the necessary adaptations to treat the non-unimodular case. Since Fourier theory is basic, the key is the non-unimodular version of the Plancherel theorem of Tatsuuma~\cite{Tat} and Duflo and Moore~\cite{DM}, as presented in book-form by H. F\"uhr in~\cite{Fu}.

\smallskip
So let $\G$ be a locally compact group with unit $\e$ and unitary dual $\widehat\G$\,. It will be assumed that our groups are second countable and of type I. The formula
\begin{equation}\label{eq:st}
\left[{\sf Op}(A)u \right](x) = \int_\G\!\int_{\widehat\G} \mathrm{Tr}_\xi\!\left(  A(x,\xi){\rm D}_\xi^{\frac 1 2} {\pi_\xi}(xy^{\scriptscriptstyle -1})^*\right)\!\Delta(y)^{- \frac 1 2} u(y)d\wm(\xi)d\m(y)
\end{equation}
is the starting point for a global pseudo-differential calculus on $\G$\,. It involves suitable operator-valued symbols $A$ defined on $\G\times \widehat\G$\,, the modular function $\Delta$ of the group and the formal dimension operators ${\rm D}_\xi$ introduced by Duflo and Moore~\cite{DM}. Formula (\ref{eq:st}) makes use of
the Haar $\m$ and Plancherel $\wm$ measures on $\G$ and $\widehat\G$ respectively. We also fixed a measurable field of irreducible representations $(\pi_\xi)_{\xi \in \widehat\G}$ such that $\pi_\xi$ belongs to the class $\xi$ and $\pi_\xi$ acts on a Hilbert space $\mathcal{H}_\xi$\,. The trace ${\rm Tr}_\xi$ refers to this Hilbert space.

\smallskip 
Formula (\ref{eq:st}) is a generalization of
the relation derived in~\cite[(1.1)]{MR} for unimodular groups, with a difference on the order of the factors that has to do with the choice of a convention for the Fourier transform. Thus our quantization will cover right invariant operators whereas the one in~\cite{MR} gives rise to left invariant operators. 

\smallskip
One of the advantages of using operator valued symbols is that one gets a global approach and a full symbol, free of localization choices, everything relying on harmonic analysis concepts attached to the group. Even for compact Lie groups there is no notion of full scalar-valued symbols for a pseudo-differential operator using local coordinates. For a more detailed discussion, for motivations and a full development of particular cases see~\cite{FR,RT,MR}. All these contain historical background and references to the existing literature treating pseudo-differential operators and quantization in a group theoretic context.

\smallskip
The first part of the article presents the quantization in the non-unimodular case. After reviewing the general form of Plancherel's Theorem, in Section 3 the basic constructions are indicated. A natural Wigner transformation appears as a dequantization, i.e. as an inverse of the ${\sf Op}$ procedure. The integral kernel of ${\sf Op}(A)$ is computed in terms of the symbol $A$\,. There are also some brief remarks upon the connections with the Schr\"odinger representation of suitable crossed product $C^*$-algebras and with families of (unbounded) Weyl operators, as well as on the ${\sf Op}$-calculus for product groups and for the Abelian case. 

\smallskip
In Section 4 we state and prove a covariance property, showing how conjugation with group translations is reflected at the level of the symbols. 

\smallskip
Until now, the formalism works mainly at the level of (suitable) $L^2$-spaces, as required by Plancherel's Theorem. Section 5 is dedicated to restrictions and extensions. The treatment of~\cite[Sect.\,5]{MR}, involving ``smooth'' compactly supported functions and corresponding distributions is briefly reviewed, since it does not depend on unimodularity. (For groups without an explicit Lie structure Bruhat's space $\mathcal D(\G)$ is useful.) Of course one is also interested in Schwartz-type spaces and their duals, formed of tempered distributions. Besides the well-known constructions for nilpotent groups (that are all unimodular), there is a less known approach valid for solvable group~\cite{DG}, many of them being non-unimodular. We show that such spaces may be used in the setting of the global quantization, for extension or restriction purposes, extending known results from the case $\R^n$.

\smallskip
One could complain that the unitary dual and its Plancherel measure and transform are complicated and abstract objects. Thus we arrive at the second main goal of this article. We put into evidence some situations when more explicit and manageable versions are available and adapt the global quantization to this concrete setting.

\smallskip
The main idea in Section 6 is to use Kirillov theory, having as a starting point the Kirillov map $\kappa:\g^\sharp/_{\!\G}\to\wG$\,, trying to involve the space $\g^\sharp/_{\!\G}$ of all coadjoint orbits in the global pseudo-differential formalism. This works well especially when the two spaces $\wG$ and $\g^\sharp/_\G$ are not too different. In particular, when $\G$ is an exponential group, Kirillov's map is a homeomorphism (hence a Borel isomorphism)~\cite{LL} and it is possible to replace $(\wG,\wm)$ with an isomorphic copy $(\g^\sharp/_{\!\G},\nu)$ in the constructions of the preceding section. The paper~\cite{DR} supplies information about the measure $\nu$. Thus one gets {\it an orbital form of the quantization}, having its covariance properties. This may also be adapted to more general situations, being enough to have a good Kirillov correspondence on controlled subsets with negligible complements. But the fact is that the orbit space and its measure theory are still intricate and non-explicit if no extra information is used.

\smallskip
Fortunately, there are cases in which the Plancherel transform has been worked
out in concrete terms. We shall review some of such situations in Section 7 and
put our quantization in this perspective. One is mainly interested in
stratifying the dual vector space $\g^\sharp$\,, putting into evidence ``the main
layer'' $\Omega\subset\g^\sharp$ as well as a suitable cross section $\Si$\,, and using as much
Euclidean theory as one can. Ideally, $\Si$ should be a smooth manifold,
explicitly modeled over an open subset of a vector space, also hoping to convert
the Plancherel measure to one equivalent to the Lebesgue measure and to compute
its density. Then {\it a concrete form of the global quantization} would be
available. Nilpotent groups are quite easy to deal with~\cite{CG}, but they are
unimodular and we would like to be able to cover more general situations. We
rely on results of Duflo-Ra\"{\i}s~\cite{DR} and Currey~\cite{Cu,Cu1,Cu2,CP} for completely solvable groups and indicate briefly the more complicated situation of exponential groups.

\smallskip
Currey's parametrizations, valid for such large classes of groups, are remarkable, actually explicit, but quite involved. But in particular cases it might be easier to get manageable parametrizations using directly Mackey's machine of induced representations. This has been done in~\cite{AV} for the $3$-dimensional connected simply connected solvable Lie groups in the Bianchi classification. In Section 8 we review their results and put them in the perspective of the concrete global pseudo-differential calculus for such groups.

\smallskip
A parameter $\tau$ (actually a measurable map $\tau:\G\to\G$) can be used in formulas to govern ordering issues, as a number  belonging to $[0,1]$ was traditionally  involved~\cite{Sh} for $\tau$-quantizations in $\R^n$. It has already been introduced in~\cite[Sect.\,4]{MR} for non-unimodular groups, its role has been analyzed and it has been shown when a symmetric Weyl-type quantization (corresponding to $\tau=1/2$ for $\R^n$) is available. We also make use of this parameter in Section 3, to convince the reader that it can be included even for non-unimodular groups. But we send to~\cite{MR} for explanations and results (especially concerning the adjoint operation at the level of symbols and the ordering of multiplications and convolution operators in this calculus). And in the second half of the article, for simplicity, we stick to the Kohn-Nirenberg-type quantization, corresponding to $\tau(\cdot)=\e$\,.

\smallskip
The basic facts and constructions of Section 3 and an analysis of the ``$ax+b$'' group appeared as an announcement in~\cite{MS}.

\smallskip
Let us make some final comments. Although the idea to involve the irreducible representation theory of a (suitable) group in the construction of a global pseudo-differential calculus appeared in the eighties, a large part of the constructions and results of the present paper are formally new. This is especially true for the concrete versions, appearing in Sections~\ref{intkid},~\ref{csoltkid} and~\ref{cracit}. Up to our knowledge, the basic formulae~\eqref{opp},~\eqref{oppad},~\eqref{oppadd},~\eqref{operatoruku},~\eqref{operatoraku}, all needing consistent explanations, did nor appear explicitly in the literature.  As soon as one states clearly the main aims of the formalism, the development follows rather straightforwardly. In particular, the proofs are few and follow quite obvious lines. Therefore, the present work could mainly be considered an expository article. But the Harmonic Analysis tools staying behind the constructions are sophisticated and deep, and putting them to work rigorously in the quantization setting requires a non-trivial effort. This is particularly the case for non-unimodular groups, the main topic of the paper. Collecting together and adapting the constructions needed for the parametrized versions in the exponential (and especially in the completely solvable case) is also a consistent work, not easily accessible for an expert in pseudo-differential operators. The general types of groups we consider do not allow a more detailed unified analytical development of the pseudo-differential calculus. In the future, having the general theory exposed here as a starting point, we intend to study deeper aspects and applications for interesting subclasses of groups.

\section{The unitary dual and the Plancherel transform}\label{bogart}

We start with a couple of notations. For a (complex, separable) Hilbert space $\H$\,, one denotes by $\mathbb B(\H)$ the $C^*$-algebra of all linear bounded operators in $\H$ and by $\mathbb K(\H)$ the closed two-sided $^*$-ideal of all the compact operators. The Schatten-von Neumann class  $\mathbb B^p(\H)$ of order $p\in[1,\infty)$ is a Banach space. Hilbert-Schmidt operators form a two-sided $^*$-ideal $\mathbb B^2(\H)$ and a Hilbert space with the scalar product $\<A,B\>_{\mathbb B_2(\H)}\!:={\rm Tr}\!\(AB^*\)$\,. 

\smallskip
Let $\G$ be a locally compact group with unit $\e$ and fixed left Haar measure $\m$\,. Being mainly interested in non-unimodular cases,  we recall the role played by {\it the modular function} for substitution of variables
\begin{equation}\label{eq:change of variables}
  \int_\G f(y)\,d\m(y) = \int_\G \Delta(x) f(yx)\,d\m(y) = \int_\G\Delta(y)^{-1} f(y^{-1})\,d\m(y)\,.
\end{equation}
There is a Banach $^*$-algebra structure on $L^1(\G)$: the convolution of two functions is defined by 
\begin{equation*}\label{convol}
(f \star g)(x) = \int_\G f(y) g(y^{-1}x)\,d\m(y)\,,
\end{equation*}
and the involution is given by
\begin{equation*}\label{acea}
f^{\star_1}(x) = \Delta(x)^{- 1} \overline{f (x^{-1})}\,.
\end{equation*}
In general, one has a $p$-dependent isometric involution on $L^p(\G)\equiv L^p(\G;\m)$ given by
\begin{equation*}\label{eq:involution}
f^{\star_p}(x) = \Delta(x)^{-\frac 1 p} \overline{f (x^{ -1})}\,.
\end{equation*}

\begin{Remark}\label{aceea}
{\rm The correspondence $L^2(\G)\ni f\to f^{\star_2}\in L^2(\G)^\dag$ is a linear unitary map. In the following we reserve the notation $f^\star\equiv f^{\star_2}$ for functions in the Hilbert space $L^2(\G)$\,. }
\end{Remark}

Let now $\G$ be a type I second countable locally compact group. We set $\,\wG:={\rm Irrep(\G)}/_{\cong}$ (cf.~\cite{Di,Fo1}) for all the classes of irreducible strongly continuous unitary representations of $\G$ and call it {\it the unitary dual of $\G$}\,. Both ${\rm Irrep(\G)}$ and $\,\wG$ are endowed with (standard) Borel structures~\cite[Sect.\,18.5]{Di}. The structure on $\wG$ is the quotient of that on ${\rm Irrep(\G)}$ and is called {\it the Mackey Borel structure}. There is a measure on $\wG$\,, called {\it the Plancherel measure associated to $\m$} and denoted by $\wm$~\cite[Sect.\,18.8]{Di}.

\smallskip
For non-unimodular groups an important role is played by {\it the Duflo-Moore operators} (also called {\it formal dimension operators}). They are densely defined positive operators with dense image ${\rm D}_\pi: \mathrm{Dom}({\rm D}_\pi) \to \mathcal H_\pi$ and satisfy almost everywhere {\it the semi-invariance condition }
\begin{equation}\label{eq:invrel}
\pi(x){\rm D}_\pi \pi(x)^* = \Delta(x)^{-1}{\rm D}_\pi, \quad\forall\,x \in\G\,.
\end{equation}

Let $\wm$ be a Plancherel measure  in $\widehat\G$\,, $(\pi_\xi)_{\xi \in \widehat\G}$ a measurable field of representations  and ${\rm D}_{\pi_\xi}\!\equiv{\rm D}_\xi : \mathcal H_\xi \to \mathcal H_\xi$ a family of densely defined positive operators  satisfying (\ref{eq:invrel}) for $\wm$-a. e.\@ $\xi \in \widehat\G$\,. We define (in the weak sense) the \textit{operator-valued Fourier transform} of a function $w\in L^1(\G)$ as the map
$$
\wG\ni\xi\mapsto\pi_\xi(w) := \int_\G w(y){\pi_\xi}(y) \,d\m(y)\in\mathbb B(\H_\xi)
$$
and the \textit{Plancherel transform} of $w \in L^1(\G)\cap L^2(\G)$ as the operator
$$
(\mathscr Pw) (\xi)\equiv\widehat w(\xi) = \pi_\xi (w) {\rm D}_\xi ^{1/2}.
$$

In such terms, the main result, that we are going to use repeatedly, is the following form~\cite{DM,Fu,Tat} of the non-commutative Plancherel theorem (the unimodular case can be found in~\cite{Di}):

\begin{Theorem}\label{sec:four-planch-transf}
Let $\,\G$ be a type I second countable locally compact group. There exists a $\sigma$-finite Plancherel measure $\wm$ on $\widehat\G$\,, a measurable field of irreducible representations $(\pi_\xi)_{\xi \in \widehat\G}$ with $\pi_\xi \in \xi$\,, a measurable field $({\rm D}_\xi)_{\xi \in \widehat\G}$ of densely defined positive operators on $\mathcal{H}_\xi$ with dense image, satisfying (\ref{eq:invrel}) for $\wm$-almost every $\xi \in \widehat\G$\,, which have the following properties:
\begin{enumerate}
\item Let $w \in L^1(\G) \cap L^2(\G)$\,. For $\wm$-almost all $\xi \in \widehat\G$\,, the operator $\widehat w(\xi)$ extends to a Hilbert-Schmidt operator on $ \mathcal H_\xi$ and
\begin{equation*}\label{eq:V_u,v}
\left\Vert\,w\,\right\Vert ^2_2 = \int_{\widehat\G}\,\left\Vert\,\widehat w(\xi)\,\right\Vert_{\mathbb B_2}^2\!d\wm(\xi)\,.
\end{equation*}
\item The Plancherel transformation extends in a unique way to a unitary operator
\begin{equation*}
\mathscr P: L^2(\G) \to \int^{\oplus}_{\widehat\G}\!\mathbb B _2(\mathcal H _\xi)\, d\wm(\xi)\,.
\end{equation*}
\item The Plancherel measure and the operator field satisfy the inversion formula
\begin{equation*}\label{eq:plancherel inversion formula}
w(x) = \int_{\widehat\G} {\rm Tr}_\xi\!\left(\widehat w(\xi){\rm D}_\xi^{\frac 1 2}{\pi_\xi}(x)^*\right)d\m(\xi)\,,
\end{equation*}
for all $w$ in the Fourier algebra of $\,\G$ (in particular, for all $w\in\mathcal D(\G)$)\,; see Section~\ref{ferfelic}. The integral converges absolutely in the sense that $\widehat w(\xi) {\rm D}_\xi^{\frac 1 2}$
extends to a trace-class operator $\wm$-a.e.~and the integral of the trace-class norms is finite.
\end{enumerate}
\end{Theorem}

We make use of the following notations:
\begin{equation*}\label{puicica}
\mathscr B^2(\wG)= \int^{\oplus}_{\widehat\G} \mathbb B_2(\mathcal H_\xi) \,d\wm(\xi)\,, \quad\mathscr{B}^1(\wG) = \int^{\oplus}_{\widehat\G} \mathbb B^1(\mathcal H_\xi) {\rm D}_\xi^{- \frac 1 2}\, d\wm(\xi)\,,
\end{equation*}
\begin{equation}\label{puicika}
\mathscr{B}^2(\G\times\wG)= L^2(\G) \otimes \mathscr B^2(\wG)\,, \quad\mathscr{B}^2(\widehat\G\times\G)= \mathscr B^2(\wG)\otimes L^2(\G)\,.
\end{equation}
$\mathscr{B}^2(\G\times\wG)$\,, one of our natural spaces of symbols, has the inner product 
$$
\langle A,B \rangle_{\mathscr{B}^2} = \int_\G \int_{\widehat\G}\mathrm{Tr}_\xi \left[A(x,\xi)B(x,\xi)^* \right]d\wm(\xi)\,d\m(x)\,.
$$

\section{The basic definitions: the global quantization and the Wigner transform}\label{aici}

We introduce pseudo-differential operators through $\tau$-quantizations for an arbitrary measurable function $\tau:\G\to\G$\,. For this we fix a Plancherel measure $\m$ and a measurable field $(\pi_\xi,{\rm D}_\xi)_{\xi\in \widehat\G}$ as in the non-commutative Plancherel Theorem. It is easy to see that different choices give rise to equivalent constructions.

\smallskip
We recall that to elements $u,v$ of $L^2(\G)$ one associates the rank one operator
\begin{equation}\label{rankone}
\Upsilon_{u,v}:L^2(\G)\to L^2(\G)\,,\quad \Upsilon_{u,v}(w):=\<w,\overline v\>_{L^2(\G)}\,u\,.
\end{equation}
This induces an isomorphism $\Upsilon:L^2(\G)\otimes L^2(\G)\to\mathbb B^2\big[L^2(\G)\big]$\,, associating integral operators to kernels
\begin{equation}\label{intket}
[\Upsilon(K)u](x):=\int_\G K(x,y)u(y)d\m(y)\,,
\end{equation}
such that $\Upsilon(u\otimes v)=\Upsilon_{u,v}$\,.

\smallskip
The remaining part of the construction, relying on the fact that $\G$ is a second countable type I group, is conveniently summarised in the following commutative diagram:
\begin{equation}\label{diagrama}
\begin{diagram}
\node{L^2(\G)\otimes L^2(\G)} \arrow{e,t}{{\sf id}\otimes{\fscr P}} \arrow{s,l}{{\sf C}^\tau}\arrow{se,r}{{\sf Sch}^\tau}\node{L^2(\G)\otimes\mathscr B^2(\wG)}\arrow{s,l}{{\sf Op}^\tau}\arrow{e,t}{\mathscr P\otimes\mathscr P^{-1}}\node{\mathscr B^2(\wG)\otimes L^2(\G)} \\ 
\node{L^2(\G)\otimes L^2(\G)} \arrow{e,b}{\Upsilon} \node{\mathbb B^2\big[L^2(\G)\big]}\arrow{n,r}{{\sf Wig}^\tau}\arrow{ne,r}{{\sf FWig}^\tau}
\end{diagram}
\end{equation}
which will be seen to be composed of isomorphisms. This requires a definition and some explanations.

\begin{Definition}\label{toate}
\begin{enumerate}
\item[(a)]
The ${\sf Op}^\tau\!$-arrow  is called {\rm the global $\tau$-pseudo-differential calculus} or {\rm the global $\tau$-quantization}, it points downwards and it is the inverse of the ${\sf Wig}^\tau\!$-arrow ({\rm the $\tau$-Wigner transformation}) pointing upwards. Thus ${\sf Wig}^\tau\!=\big({\sf Op}^\tau\big)^{-1}\!=\big({\sf Op}^\tau\big)^*\!$ will be seen as a dequantization procedure (a symbol map).
\item[(b)]
We call ${\sf FWig}^\tau:=\big(\mathscr P\otimes\mathscr P^{-1}\big)\!\circ{\sf Wig}^\tau$ {\rm the $\tau$-Fourier-Wigner transformation}.
\item[(c)]
We call ${\sf Sch}^\tau\!=\Upsilon\circ{\sf C}^\tau$ {\rm the $\tau$-Schr\"odinger representation}.
\end{enumerate}
\end{Definition}

We now introduce the ``change of variable'' transformation ${\sf C}^\tau$. Given a square integrable function $K:\G\times\G\to\mathbb C$, we set
\begin{equation} \label{lipsa}
\big[{\sf C}^\tau\!(K)\big](x,y):=\Delta(y)^{-\frac 1 2}\,K\big(\tau(yx^{-1})x,xy^{-1}\big)\,.
\end{equation}

\begin{Proposition}\label{ofi}
The transformations ${\sf C}^\tau,{\sf Sch}^\tau,{\sf Op}^\tau,{\sf Wig}^\tau,{\sf FWig}^\tau$ are unitary.
\end{Proposition}

\begin{proof}
Recall that $\Upsilon$ and ${\fscr P}$ are unitary. We show that ${\sf C}^\tau: L^2(\G)\to L^2(\G)\otimes L^2(\G)$ is also unitary. Then this will define the unitary maps ${\sf Sch}^\tau,{\sf Op}^\tau,{\sf Wig}^\tau,{\sf FWig}^\tau$ making the diagram commute.

\smallskip
The fact that ${\sf C}^\tau$ is well-defined and isometric follows easily from Fubini's theorem, the left invariance of the Haar measure and formula~\eqref{eq:change of variables}. Then one also checks by direct computations that the adjoint of ${\sf C}^\tau$ coincides with its inverse and is given by
\begin{equation}\label{pedos}
\big[({\sf C}^\tau)^{-1}(L)\big](x,y) = \Delta\big(y^{-1}\tau(y^{-1})^{-1}x\big)^{ \frac 1 2}L\big(\tau(y^{-1})^{-1}x,y^{-1}\tau(y^{-1})^{-1}x\big)\,.
\end{equation}
\end{proof}

\begin{Remark}\label{frightening}
{\rm The formulae~\eqref{lipsa} and~\eqref{pedos} look quite intricate. But let us suppose that $\G$ is Abelian, in additive notation, and that $\tau:\G\to\G$ commutes with inversion (being a homomorphism, for instance). Then the two formulae read
\begin{equation} \label{kerneltau}
[{\sf C}^\tau\!(K)](x,y) = K(x+\tau(y-x),x-y)\,,\quad \big[({\sf C}^\tau)^{-1}(L)\big](x,y) = L(x+\tau(y),x-y+\tau(y))\,.
\end{equation}
If $\G=\R^n$ is a vector space, multiplication with a number $\tau\in[0,1]$ fits in our scheme and the formulae are familiar from the theory of $\tau$-quantizations~\cite{Sh}, important particular cases being $\tau=0$ (Kohn-Nirenberg), $\tau=1/2$ (Weyl) and $\tau=1$ (right quantization). In our framework too, the mapping $\tau$ is related to ordering issues. We refer to~\cite[Sect. 4 and 6]{MR} for unimodular groups and to~\cite[Sect. 6]{MS} for discussions about how the formalism covers multiplication and convolution operators,  the way these are ordered and Weyl-type requirements on $\tau$ to have a simple condition on the symbol $A$ yielding a self-adjoint operator ${\sf Op}^\tau\!(A)$\,. The relation between different $\tau$-quantizations is similar to the unimodular case~\cite[Sect 3]{MR}.
}
\end{Remark}

\begin{Remark}\label{simpleone}
{\rm In the important simple case $\tau(x)=\e\,$ for every $x\in\G$ we will skip the upper index, writting ${\sf C},{\sf Sch},{\sf Op},{\sf Wig},{\sf FWig}$\,. One could call ${\sf Op}$ {\it the Kohn-Nirenberg global quantization}.
}
\end{Remark}

Making use of the diagram~\eqref{diagrama} and of the notation~\eqref{puicika}, we have obtained unitary transformations
\begin{equation}\label{finop}
{\sf Op}^\tau\!:=\Upsilon\circ{\sf C}^\tau\!\circ({\sf id}\otimes\mathscr P)^{-1}\!:\mathscr B^2(\G\times\wG)\to\mathbb B^2\big[L^2(\G)\big]
\end{equation} 
(global $\tau$-quantization) and its inverse (dequantization)
\begin{equation*}\label{finwig}
{\sf Wig}^\tau\!:=({\sf id}\otimes\mathscr P)\circ({\sf C}^\tau)^{-1}\!\circ\Upsilon^{-1}:\mathbb B^2\big[L^2(\G)\big]\to\mathscr B^2(\G\times\wG)\,.
\end{equation*} 
${\sf Op}^\tau\!(A)$ is then called the \emph{$\tau$-pseudo-differential operator}\index{pseudo-differential operator} with symbol $A$\,. Unitarity implies
\begin{equation}\label{suntunitar}
\big\<{\sf Op}^\tau\!(A),T\big\>_{\mathbb B^2[L^2(\G)]}=\big\<A,{\sf Wig}^\tau\!(T)\big\>_{\mathscr B^2(\G\times\wG)}\,,\quad\forall\,A\in\mathscr B^2(\G\times\wG)\,,\,T\in\mathbb B^2[L^2(\G)]\,.
\end{equation}
One computes easily
\begin{equation}\label{opp}
\big[{\sf Op}^\tau\!(A)u\big](x)= \int_\G \int_{\widehat\G} \mathrm{Tr}_\xi\!\left(A(\tau(yx^{ -1})x,\xi){\rm D}_\xi^{\frac 1 2}{\pi_\xi}(yx^{-1}) \right)\!\Delta(y)^{-\frac 1 2}u(y) \,d\wm(\xi)\,d\m(y)\,.
\end{equation}
Thus, the integral kernel of ${\sf Op}^\tau\!(A)$ is the square integrable function $\ker_A^\tau={\sf C}^\tau\big[({\sf id}\otimes\mathscr P)^{-1}A\big]$ given by
\begin{align*}\label{eq:ker}
\ker_A^\tau(x,y)&= \Delta(y)^{-\frac 1 2}\!\int_{\widehat\G}{\rm Tr}_\xi\!\left(A(\tau(yx^{ -1})x,\xi){\rm D}_\xi^{\frac 1 2}{\pi_\xi}(xy^{-1})^*\right)d\wm(\xi)\,. 
\end{align*}
Writing ${\sf Wig}^\tau_{u,v}:={\sf Wig}^\tau\!(\Upsilon_{u,v})$\,, we get the quantization of rank one operators~\eqref{rankone}
\begin{equation*}\label{oanrenc}
\Upsilon_{u,v}={\sf Op}^\tau\!\big({\sf Wig}^\tau_{u,v}\big)\,,
\end{equation*}
where
\begin{equation}\label{wigg}
{\sf Wig}^\tau_{u,v}(x,\xi)=\int_\G\,\Delta\big(y^{-1}\tau(y^{-1})^{-1}x\big)^{ \frac 1 2}u\big(\tau(y^{-1})^{-1}x)v(y^{-1}\tau(y^{-1})^{-1}x\big){\rm D}_\xi^{1/2}\pi_\xi(y)^*d\m(y)\,.
\end{equation}
A different but equivalent realization is achieved via the Fourier-Wigner transformation, taking on elementary vectors $u\otimes v$ the explicit form (note that the role of the variables $x,y$ is different from~\eqref{wigg})
\begin{equation*}\label{frigg}
{\sf FWig}^\tau_{u,v}(\xi,y)=\int_\G\,\Delta\big(y^{-1}\tau(y^{-1})^{-1}x\big)^{ \frac 1 2}u\big(\tau(y^{-1})^{-1}x)v(y^{-1}\tau(y^{-1})^{-1}x\big){\rm D}_\xi^{1/2}\pi_\xi(x)^*d\m(x)\,.
\end{equation*}
Setting $\,\widehat A:=\big(\mathscr P\otimes\mathscr P^{-1}\big)A$\,, the case $T=\Upsilon_{u,v}$ in~\eqref{suntunitar} then reads
\begin{equation}\label{rids}
\big\<{\sf Op}^\tau\!(A)u,v\big\>_{L^2(\G)}=\big\<A,{\sf Wig}^\tau_{u,v}\big\>_{\mathscr B^2(\G\times\wG)}=\big\<\widehat A,{\sf FWig}^\tau_{u,v}\big\>_{\mathscr B^2(\wG\times\G)}\,.
\end{equation}

\begin{Remark}\label{ofi}
{\rm In the diagram~\eqref{diagrama} we also included the Schr\"odinger representation ${\sf Sch}^\tau\!:=\Upsilon\circ{\sf C}^\tau$ defined for $\,K\in L^2(\G\times\G)$ by
\begin{equation}\label{radar}
\[{\sf Sch}^\tau\!(K)v\]\!(x):=\int_\G\!\Delta(y)^{-1/2}K\!\left(\tau(yx^{-1})x,xy^{-1}\right)\!v(y)\,d\m(y)\,. 
\end{equation}
For unimodular groups, it is shown in~\cite[Sect.\,7]{MR} how this is related to the theory of crossed product $C^*$-algebras associated to the action ${\rm left}$ of the group $\G$ by left translations on $C^*$-algebras $\A$ of complex functions defined on $\G$ (or more general). Actually~\eqref{radar} is the integrated form of a canonical (Schr\"odinger) covariant representation of the dynamical system $(\A,{\rm left},\G)$ and it is a representation in the Hilbert space $L^2(\G)$ of both the full and the reduced crossed products $\A\!\rtimes_{{\rm left}}\!\G$ and $\A\!\rtimes_{{\rm left}}^{\rm red}\!\G$\,. The connection to global pseudo-differential operators is given by applying the partial Plancherel transform. Besides yielding new $\A$-dependent symbol spaces for the quantization (the projective tensor product $L^1(\G)\overline\otimes\A$ is densely contained in both crossed products) and allowing a neat identification of the compact operators (those corresponding to the case $\A=C_0(\G)$), this formalism is also useful to provide Fredholm and spectral properties of pseudo-differential operators, as it was done in~\cite{Ma} for unimodular groups. The non-unimodular ones can be treated very similarly, so we will not give a detailed approach.
}
\end{Remark}

\begin{Remark}\label{likely}
{\rm For direct products $\G:=\G_1\times\G_2$ of type I second countable locally compact groups one has, after natural identifications,  
$$
{\sf Op}^{\tau_1\times\tau_2}_{\G_1\times\G_2}\!={\sf Op}^{\tau_1}_{\G_1}\!\otimes{\sf Op}^{\tau_2}_{\G_2}\,.
$$ 
For this one has to assume that the quantization parameter $\tau$ has the form $\tau_1\times\tau_2$\,, leading to ${\sf C}_\G^\tau={\sf C}_{\G_1}^{\tau_1}\otimes{\sf C}_{\G_2}^{\tau_2}$ (after suitable unitary identifications based on $(\G_1\times\G_2)\times(\G_1\times\G_2)\equiv(\G_1\times\G_1)\times(\G_2\times\G_2$))\,. The assertion is easy to verify starting from~\eqref{finop}. In~\cite[Sect.\,7.25]{Fo1} it is established that, in the type I setting, $\widehat{\G_1\times\G_2}$ can be identified with $\wG_1\times\wG_2$\,, and the Plancherel data on $\G_1\times\G_2$ (measure and Duflo-Moore family) decomposes as a tensor product of Plancherel data of the two groups (use unicity properties).
}
\end{Remark}

\begin{Remark}\label{liliky}
{\rm When $\G$ is Abelian (unimodular, in particular), its unitary dual $\wG$ is also an Abelian locally compact group in a natural way, Plancherel measures are Haar measures of this group and $\mathscr B^2(\wG)=L^2(\wG)$\,. By Pontryagin theory, the second dual is canonically equivalent to $\G$ and thus 
$$
\wG\times\widehat{\wG}\cong\wG\times\G\cong\G\times\wG\,.
$$ 
This can be used to show that in this case the global quantizations associated to $\G$ and to $\wG$ are equivalent.
}
\end{Remark}

{\bf Summary.} Although one can proceed with the $\tau$-quantizations, from now on we restrict to the basic case $\tau(\cdot)=\e$\,. For the convenience of the reader, and for further reference, we summarize.

\begin{equation}\label{operatoru}
[{\sf Op}(A)u](x)=\int_\G \int_{\widehat\G} \mathrm{Tr}\!\left(A(x,\xi){\rm D}_\xi^{\frac 1 2}{\pi_\xi}(yx^{-1}) \right)\!\Delta(y)^{-\frac 1 2}u(y) \,d\wm(\xi)\,d\m(y)\,,
\end{equation}
\begin{equation*}\label{radiar}
\[{\sf Sch}(K)v\]\!(x)=\int_\G\!\Delta(y)^{1/2}K\!\left(x,xy^{-1}\right)\!v(y)\,d\m(y)\,, 
\end{equation*}
\begin{equation*}\label{functiaker}
\ker_A(x,y)= \Delta(y)^{-\frac 1 2} \int_{\widehat\G}{\rm Tr}\left( a(x,\xi){\rm D}_\xi^{\frac 1 2}{\pi_\xi}(xy^{-1})^*\right)d\wm(\xi)\,, 
\end{equation*}
\begin{equation*}\label{functiav}
{\sf Wig}_{u,v}(x,\xi) = \int_\G \Delta\big(y^{ -1}x\big)^{\frac 1 2}\,u\big(x\big)\,v\big(y^{ -1}x\big)\,{\rm D}_\xi^{\frac 1 2}{\pi_\xi}(y)^*\,d\m(y)\,,
\end{equation*}
\begin{equation*}\label{functiaw}
{\sf FWig}_{u,v}(\xi,x)= \int_\G\,\Delta\big(x^{-1}y\big)^{ \frac 1 2}u\big(y)v(x^{-1}y\big){\rm D}_\xi^{\frac{1}{2}}\pi_\xi(y)^*d\m(y)\,\,.
\end{equation*}

\begin{Remark}\label{ujanschi}
{\rm One may introduce the notion of a {\it Weyl system}; this is then used to recover pseudo-differential operators. This has been done in~\cite{MR} in the unimodular case; the elements of the Weyl system were unitary operators, expressing basic commutation relations in exponentiated form. We indicate briefly and formally the corrections needed for non-unimodular groups, leading this time to unbounded operators. This will not be needed or developed below.

\smallskip
We try to write for $u,v\in L^2(\G)\,, x\in \G\,,\xi\in\wG$ and $\phi\in\H_\xi$
\begin{equation*}\label{midel}
{\sf FWig}_{u,v}(\xi,x) \phi = \int_\G \left[{\rm FW}(\xi,x) \,(u \otimes \phi)\right]\!(y) v(y) \,d\m(y)\,,
\end{equation*}
\begin{equation*}\label{modell}
{\sf Wig}_{u,v}(x,\xi)\phi = \int_\G \left[{\rm W}(x,\xi) \,(u \otimes \phi)\right]\!(y) v(y) \,d\m(y)\,,
\end{equation*}
for operators ${\rm FW}(x,\xi)$ and ${\rm W}(x,\xi)$ in $L^2(\G;\H_\xi)$\,. This would lead, at least formally, to formulae as
\begin{equation*}\label{baraon}
{\sf Op}(A)=\int_\G\int_{\wG}{\rm Tr}_\xi[A(x,\xi){\rm W}(x,\xi)]d\m(x)d\wm(\xi)=\int_\G\int_{\wG}{\rm Tr}_\xi[\widehat A(\xi,x){\rm FW}(\xi,x)]d\m(x)d\wm(\xi)\,.
\end{equation*}

A short computation shows that
\begin{equation*}\label{radup}
\big[{\rm W}_{u,v}^\tau(x,\xi)\Theta\big](y)=\Delta(y)^{-1/2}\big[{\rm D}^{1/2}_\xi\pi_\xi(yx^{-1})\big]\big(\Th(x)\big)\,,
\end{equation*}
\begin{equation*}\label{radalup}
\big[{\rm FW}_{u,v}^\tau(\xi,x)\Theta\big](y)=\Delta(y)^{1/2}\big[{\rm D}^{1/2}_\xi\pi_\xi(xy)^*\big]\big(\Th(xy)\big)\,,
\end{equation*}
and the appearance of the non-unimodular ingredients $\Delta$ and ${\rm D}^{1/2}_\xi$ makes them unbounded.
}
\end{Remark}

\begin{Remark}\label{Nguyen}
{\rm One is not bound to use irreducible representations. In certain cases we could make use of other tools, as soon as a reasonable version of the Plancherel transformation is available, meaning in loose terms a unitary equivalence defined on $L^2(\G)$\,, with explicit inverse and involving concepts from harmonic analysis. Actually, this will be seen in sections~\ref{intkid}  and~\ref{csoltkid}.
We mention briefly another example~\cite{Ng}. By definition, {\it a generalized motion group} is a semidirect product $\G:={\sf V}\rtimes{\sf K}$\,, where the compact group ${\sf K}$ acts on the Euclidean space ${\sf V}$ by orthogonal transformations. For every $\lambda\in{\sf V}$ let us define {\it the quasi-regular representation}
\begin{equation*}\label{qreg}
\pi^\lambda:\G\to\mathbb B\big[L^2({\sf K})\big]\,,\quad \big[\pi^\lambda({\sf v,k})\varphi\big]({\sf h}):=e^{2\pi i(\lambda\mid{\sf hv})}\varphi({\sf hk})\,.
\end{equation*}
Setting 
\begin{equation*}\label{pseudoplancherel}
{\sf F}:L^1(\G)\cap L^2(\G)\to L^2\big({\sf V};\mathbb B\big[L^2({\sf K})\big]\big)\,,\quad [{\sf F}(u)](\lambda):=\int_\G u({\sf v,k})\pi^\lambda({\sf v,k})d{\sf v}d{\sf k}\,,
\end{equation*}
one gets by extension an isomorphism ${\sf F}:L^2(\G)\to L^2\big({\sf V};\mathbb B\big[L^2({\sf K})\big]\big)$ and an inversion formula, leading to a global quantization 
$$
{\sf OP}:L^2(\G)\otimes L^2\big({\sf V};\mathbb B\big[L^2({\sf K})\big]\big)\to\mathbb B^2\big[L^2(\G)\big]
$$ 
given formally by
\begin{equation}\label{foloform}
[{\sf OP}(a)u]({\sf v,k}):=\int_\G\int_{\sf V} {\rm Tr}\big[a({\sf v,k};\lambda)\pi^\lambda\big({\sf v'-k'k^{-1}v,k'k^{-1}}\big)\big]u({\sf v',k'})d\lambda d{\sf v}'d{\sf k'}.
\end{equation}
We learned this approach from~\cite{Ng}, where the global quantization is developed in a detailed pseudo-differential calculus with H\"ormander-type symbol classes. The rather simple ingredients in~\eqref{foloform} are essential for the success of some of the analytical aspects of the approach. {\it But there is no direct use of the unitary dual in this case.} The quasi-regular representations $\pi^\lambda$ are irreducible if and only if $\lambda\ne 0$ and ${\sf K}$ is Abelian (which is not the most interesting case). In addition $\pi^\lambda$ and ${\pi^\lambda}'$ are actually unitarily equivalent if $\lambda$ and $\lambda'$ are in the same ${\sf K}$-orbit (the same sphere in $\R^n$ if ${\sf K}={\rm SO}(n)$)\,.
}
\end{Remark}

\section{Covariance properties}\label{secesion}

Let us set
\begin{equation}\label{unaltta}
\Pi(x):=\int^\oplus_{\wG}\!\pi_\xi(x)d\wm(\xi)\,,\quad\forall\,x\in\G\,.
\end{equation}
It is a unitary element of the von Neumann algebra $\int^\oplus_{\widehat\G}\mathbb B(\H_\xi)d\wm(\xi)$ of decomposable operators in the Hilbert space $\int^\oplus_{\widehat\G}\H_\xi d\wm(\xi)$\,. This von Neumann algebra also acts to the left and to the right on the direct integral Hilbert space $\mathscr B^2(\wG)$\,, so we may define ${\bf ad}_{\Pi(x)}:\mathscr B^2(\wG)\to\mathscr B^2(\wG)$ by the formula
\begin{equation*}\label{sofistic}
{\bf ad}_{\Pi(x)}\Big(\int^\oplus_{\widehat\G}\!T(\xi)d\wm(\xi)\Big):=\int^\oplus_{\widehat\G}\!\pi_\xi(x)T(\xi)\pi_\xi(x)^*d\wm(\xi)\,.
\end{equation*}
We also recall {\it the left regular representation} 
\begin{equation*}\label{leftregular}
{\sf Left}:\G\to\mathbb B\big[L^2(\G)\big]\,,\quad\big[{\sf Left}_z(u)\big](x):=u\big(z^{-1}x\big)
\end{equation*}
and {\it the right regular representation} 
\begin{equation*}\label{rightregular}
{\sf Right}:\G\to\mathbb B\big[L^2(\G)\big]\,,\quad\big[{\sf Right}_z(u)\big](x):=\Delta(z)^{1/2}\,u(xz)\,.
\end{equation*}

\begin{Proposition}\label{avenir}
For any $A\in\mathscr B^2\big(\G\times\wG\big)$ and $z\in\G$ one has
\begin{equation}\label{covariation}
{\sf Left}_z\circ{\sf Op}^\tau\!(A)\circ{\sf Left}_z^* = {\sf Op}^\tau\!\Big[\big({\sf Left}_z\!\otimes{\bf ad}_{\Pi(z)}\big)A\Big]\,.
\end{equation}
\end{Proposition}

\begin{proof}
One can show this starting from~\eqref{opp}, but it is more illuminating to rely on~\eqref{finop}. Setting for any unitary operator $S:L^2(\G)\to L^2(\G)$
$$
{\bf ad}_S\!:\mathbb B^2\big[L^2(\G)\big]\to\mathbb B^2\big[L^2(\G)\big]\,,\quad{\bf ad}_S(T):=S\circ T\circ S^*,
$$
one has to prove $\,{\bf ad}_{{\sf Left}_z}\!\circ{\sf Op}^\tau\!={\sf Op}^\tau\!\circ\big({\sf Left}_z\!\otimes{\bf ad}_{\Pi(z)}\big)$\,, i.e.
\begin{equation}\label{formfinal}
{\bf ad}_{{\sf Left}_z}\!\circ\Upsilon\circ{\sf C}^\tau\!\circ({\sf id}\otimes\mathscr P)^{-1}\!=\Upsilon\circ{\sf C}^\tau\!\circ({\sf id}\otimes\mathscr P)^{-1}\!\circ\big({\sf Left}_z\!\otimes{\bf ad}_{\Pi(z)}\big)\,.
\end{equation}
By using~\eqref{intket}, one gets immediately
\begin{equation}\label{lainceput}
{\bf ad}_{{\sf Left}_z}\!\circ\Upsilon=\Upsilon\circ\big({\sf Left}_z\otimes{\sf Left}_z\big)\,.
\end{equation}
Then
$$
\begin{aligned}
\Big(\big[\big({\sf Left}_z\otimes{\sf Left}_z\big)\circ{\sf C}\big]K\Big)(x,y)&=[{\sf C}(K)]\big(z^{-1}x,z^{-1}y\big)\\
&=\Delta\big(z^{-1}y\big)^{-1/2}\,K\big(z^{-1}x,z^{-1}xy^{-1}z\big)\\
&=\Delta\big(y\big)^{-1/2}\big[\big({\sf Left}_z\otimes{[\sf Left}_z{\sf Right}_z]\big)(K)\big]\big(x,xy^{-1}\big)\\
&=\Big(\big[{\sf C}\circ\big({\sf Left}_z\otimes{[\sf Left}_z{\sf Right}_z]\big)\big]K\Big)(x,y)\,,
\end{aligned}
$$
so one gets
\begin{equation}\label{lasfarsit}
\big({\sf Left}_z\otimes{\sf Left}_z\big)\circ{\sf C}={\sf C}\circ\big({\sf Left}_z\otimes{[\sf Left}_z{\sf Right}_z]\big)\,.
\end{equation}
Finally, using the definitions, the semi-invariance~\eqref{eq:invrel} of the Duflo-Moore operators and changes of variables relying on~\eqref{eq:change of variables}, one computes
$$
\begin{aligned}
\Big[\big({\mathscr P}\circ\big[{\sf Left}_z{\sf Right}_z]\big)u\Big](\xi)&=\int_\G u\big(z^{-1}xz\big)\Delta(z)^{1/2}\pi_\xi(x){\rm D}_\xi^{1/2}d\m(x)\\
&=\int_\G u(x)\Delta(z)^{-1/2}\pi_\xi\big(z)\pi_\xi(x)\pi_\xi(z)^*{\rm D}_\xi^{1/2}d\m(x)\\
&=\pi_\xi(z)\int_\G u(x)\pi_\xi(x){\rm D}_\xi^{1/2}d\m(x)\pi_\xi(z)^*\\
&=\big[\Pi(z)\mathscr P(u)\Pi(z)^*\big](\xi)\,,
\end{aligned}
$$
to be written
\begin{equation}\label{insfarshit}
{\mathscr P}\circ\big[{\sf Left}_z{\sf Right}_z]={\bf ad}_{\Pi(z)}\circ\mathscr P\,.
\end{equation}
Now~\eqref{formfinal} follows from~\eqref{lainceput},~\eqref{lasfarsit} and~\eqref{insfarshit} (in which we change $z$ into $z^{-1}$) and the proof is finished.
\end{proof}

\begin{Corollary}\label{suvenir}
For any $\,T\in\mathbb B^2\big[L^2(\G)\big]$ and $z\in\G$ one has
\begin{equation}\label{covarication}
{\sf Wig}^\tau\big({\sf Left}_z\circ T\circ{\sf Left}_z^*\big) = \big({\sf Left}_z\!\otimes{\bf ad}_{\Pi(z)}\big)\big({\sf Wig}^\tau(T)\big)\,.
\end{equation}
In particular, for every $u,v\in L^2(\G)$ one has
\begin{equation}\label{coprevarication}
{\sf Wig}^\tau_{{\sf Left}_z\!(u),{\sf Left}_z\!(v)} = \big({\sf Left}_z\!\otimes{\bf ad}_{\Pi(z)}\big)\big({\sf Wig}^\tau_{u,v}\big)\,.
\end{equation}
\end{Corollary}

\begin{proof}
Since ${\sf Wig}^\tau$ and ${\sf Op}^\tau$ are reciprocal maps,~\eqref{covarication} follows immediately from~\eqref{covariation}.

\smallskip
To get~\eqref{coprevarication} one sets $T:=\Upsilon_{u,v}$\,, uses ${\sf Wig}^\tau_{u,v}:={\sf Wig}^\tau\!(\Upsilon_{u,v})$ and the formulas $S\Upsilon_{u,v}S^*=\Upsilon_{S(u),\tilde S(v)}$ with $\tilde S(v):=\overline{S(\overline v)}$\,, noticing that ${\sf Left}_z^*={\sf Left}_{z^{-1}}$ commutes with complex conjugation.
\end{proof}

\section{Restrictions and extensions of the pseudo-differential calculus}\label{ferfelic}

It is interesting and useful to extend the quantization to more general classes of symbols and to operators which are only densely defined on $L^2(\G)$\,.

\smallskip
One option is to use {\it the Bruhat space} $\mathcal D(\G)$ and its strong dual $\mathcal D'(\G)$\,. They are defined in~\cite{Br} and are meant to extend to the locally compact group case the spaces of smooth compactly supported test functions $C_{\rm c}^\infty(\G)$ and its dual, formed of distributions, that only makes sense for Lie groups. We refer to~\cite{Br} and to~\cite[Sect.\,5]{MR} for the definitions and the basic properties; for Lie groups they have the usual meaning. In~\cite[Sect.\,5]{MR} they were used to extend the global quantization for the case of unimodular groups. The calculus being now available in general form, the same treatment is available, since the Bruhat spaces need no unimodularity condition and the new features of the non-unimodular Plancherel transform are easily taken into account. For completeness, we only state briefly the extension result. If $\G$ is a Lie group then $\mathcal D(\G)=C_{\rm c}^\infty(\G)$\,, with the usual inductive limit topology; for the general case, the reader should consult~\cite{Br}.

\smallskip
One defines $\mathscr D(\wG):={\fscr P}[\mathcal D(\G)]$\,, also transferring by the Plancherel map the locally convex topology from the Bruhat space $\mathcal D(\G)$\,. One has continuous and dense embeddings $\mathcal D(\G)\hookrightarrow\mathcal C_{\rm c}(\G)\hookrightarrow L^2(\G)$\,, so $\mathscr D(\wG)$ is a dense subspace of $\mathscr B^2(\wG)$\,. Unfortunately, $\mathscr D(\wG)$ is difficult to describe explicitly even in some of the simplest cases. By the Kernel Theorem for Bruhat spaces~\cite[Sect.\,5]{Br} one has
\begin{equation*}\label{inseninare}
\mathcal D(\G\times\G)\cong\mathcal D(\G)\,\overline\otimes\,\mathcal D(\G)\hookrightarrow L^2(\G\times\G)
\end{equation*}
continuously and densely. The symbol $\overline{\otimes}$ indicates the projective tensor product, but we recall that the spaces involved here and below are known to be nuclear. 
This refers in particular to
\begin{equation*}\label{inserare}
\mathscr D\big(\G\times\wG\big):=\mathcal D(\G)\,\overline\otimes\,\mathscr D(\wG)\hookrightarrow\mathscr B^2(\G\times\wG)\,.
\end{equation*}
Taking into account the strong dual, one gets a Gelfand triple $\mathscr D\big(\G\times\wG\big)\hookrightarrow\mathscr B^2\big(\G\times\wG\big)\hookrightarrow\mathscr D'\big(\G\times\wG\big)$\,.

\begin{Proposition}\label{quadrat}
The pseudo-differential calculus $\,{\sf Op}:L^2(\G)\otimes\mathscr B^2(\wG)\rightarrow\mathbb B^2\big[L^2(\G)\big]$ restricts to a topological isomorphism $\,{\sf Op}:\mathscr D\big(\G\times\wG\big)\rightarrow\mathbb L\big[\mathcal D'(\G),\mathcal D(\G)\big]$ and extends to a topological isomorphism $\,{\sf Op}:\mathscr D'\big(\G\times\wG\big)\rightarrow\mathbb L\big[\mathcal D(\G),\mathcal D'(\G)\big]$\,. 
\end{Proposition}

We denoted by $\mathbb L(\mathcal M,\mathcal N)$ the space of all the linear continuous operators between the topological vector spaces $\mathcal M$ and $\mathcal N$. Besides applying the Kernel Theorem, one has to check that ${\sf C}:\mathcal D(\G\times\G)\to\mathcal D(\G\times\G)$ is a linear homomorphism. If $\tau$-quantizations are considered, clearly one needs specific requirements on $\tau$ to ensure the same property for ${\sf C}^\tau$ introduced in~\eqref{lipsa}.

\begin{Remark}\label{cow}
{\rm Note that $\D(\G)$ is left invariant by the (left and right) translations. The covariance relations~\eqref{covariation} and~\eqref{covarication} remain true in the extended setting, with suitable reinterpretations.}
\end{Remark}

Extending concepts as Schwartz functions and tempered distributions to Lie groups is tricky. For connected simply connected nilpotent groups this is done quite easily by composing with the inverse of the exponential map $\exp:\g\to\G$ elements of the usual Schwartz space $\S(\g)$ of the Lie algebra. Even if for the more general exponential groups, by definition, the exponential map is a diffeomorphism, the same strategy is insufficient, since in general the resulting $\S(\G)$ would be incomparable with the Lebesgue spaces $L^p(\G)$\,. A convenient modification has been proposed in~\cite{DG} for connected simply connected solvable groups (named Solvable from now on), and it may be used for our extension purposes.

\smallskip
So let us fix a type $I$ Solvable Lie group $\G$ with Lie algebra $\g$ and a basis $\{X_1,\dots,X_N\}$ of $\g$\,. The following definition can be extracted from~\cite{DG}:

\begin{Definition}\label{david}
The Borel function $\sigma:\G\to\R_+$ is {\rm an admissible weight} if
\begin{enumerate}
\item[(i)] 
there exists $q\in\N$ such that $\,\int_\G\frac{d\m(x)}{\sigma(x)^q}<\infty$\,,
\item[(ii)]
there exists $m\in\N$ such that $\,\Delta(x)\le\si(x)^m$ for every $x\in\G$\,,
\item[(iii)]
there exist $r\in\N$ and $C\ge 1$ such that $\,C^{-1}\si(x)^{1/r}\le\si\big(x^{-1}\big)\le C\si(x)^r$ for every $x\in\G$\,,
\item[(iv)]
there exist $s\in\N$ and $C\ge 1$ such that $\,\si(xy)\le C\si(x)^s\si(y)^s$ for every $x,y\in\G$\,.
\end{enumerate}
\end{Definition}

Actually~\cite{DG} starts with the construction of an explicit function $\si$, depending on some choices and a special realization of the Solvable group $\G$\,, and then the properties above are deduced. But the subsequent results of~\cite{DG} then only depend on the properties and not on the explicit form of the weight $\si$.

\begin{Definition}\label{guillou}
An admissible weight $\si$ on the Solvable group $\G$ being given, one defines {\rm the Schwartz-type space} $\S_\si(\G)$ to be the family of all smooth complex functions $w$ on $\G$ such that all the seminorms
\begin{equation}\label{seminorms}
\p\!w\!\p_{k,\alpha}^\infty\,:=\,\p\!\si^k X^\alpha w\!\p_{L^\infty(\G)}\,,\quad k\in\N\,,\,\alpha\in\N^N
\end{equation}
are finite.
\end{Definition}

The function $\si$ being admissible, in~\cite{DG} it is shown that {\it $\S_\si(\G)$ is a nuclear Fr\'echet space with continuous dense embedings $\D(\G)\hookrightarrow\S_\si(\G)\hookrightarrow L^p(\G)$ for every $p\in[1,\infty]$}\,. See also~\cite{Sch}\,. The same topological vector space emerges if in~\eqref{seminorms} one uses $L^p$-norms instead of the uniform norm, right Haar measures instead of left Haar measures; the convention upon $X^\alpha$ (right or left invariant) is also not important. At least for the particular choice of $\si$ indicated in~\cite{DG}, one gets $\S(\g)=\S_\si(\G)\circ\exp$ in the nilpotent case, while for groups with polynomial growth $\S_\si(\G)$ can be identified with $\S(\R^N)$\,.

\smallskip
The topological dual of $\S_\si(\G)$ with the strong dual topology is denoted by $\S'_\si(\G)$\,, its elements are {\it tempered distributions}. It is shown that they are finite sums of $X^\alpha$-derivatives of continuous functions that are growing slowly at infinity with respect to the weight (i.e. satisfying $|u|\le C\si^l$ for some $l\in\N$)\,. 

\smallskip
To state our extension result, we also define $\mathscr S_\si\big(\wG\big):=\mathscr P[\S_\si(\G)]\hookrightarrow \mathscr B^2(\wG)$\,, with the Fr\'echet structure transported from $\S_\si(\G)$\,. Once again this is a space which is difficult to describe in concrete terms even in simple cases. Then we set $\mathscr S'_\si(\G\times\wG)$ for the topological dual of the (projective) tensor product $\S_\si(\G)\overline\otimes\mathscr S_\si(\wG)\hookrightarrow\mathscr B^2(\G\times\wG)$\,. 

\begin{Proposition}\label{quadrat}
The pseudo-differential calculus $\,{\sf Op}:L^2(\G)\otimes\mathscr B^2(\wG)\rightarrow\mathbb B^2\big[L^2(\G)\big]$ 
extends to a topological isomorphism $\,{\sf Op}:\mathscr S'_\si\big(\G\times\wG\big)\rightarrow\mathbb L\big[\mathcal S_\si(\G),\mathcal S'_\si(\G)\big]$\,, where the space of distributions is endowed with the strong dual topology, while the operator space carries the topology of uniform convergence on bounded subsets. 
\end{Proposition}

\begin{proof}
Recalling from Section~\ref{aici} that ${\sf Op}=\Upsilon\circ{\sf C}\circ\big({\sf id}\otimes{\fscr P}^{-1}\big)$ in the $L^2$-type spaces, every ingredient being unitary, let us take into consideration the next diagram, describing the extension:
\begin{equation*}\label{potrat}
\begin{diagram}
\node{\mathscr S'_{\!\si}(\G\times\wG)} \arrow{e,t}{\!({\sf id}\overline\otimes\fscr P)^\dag} \arrow{s,l}{{\sf Op}}\node{\mathcal S'_\Si(\G\times\G)}\arrow{s,r}{{\sf C}}\\ 
\node{\mathbb L\big[\mathcal S_\si(\G),\mathcal S_\si'(\G)\big]} \node{\mathcal S'_\Si(\G\times\G)} \arrow{w,t}{\Upsilon}
\end{diagram}
\end{equation*}

For the product group $\G\times\G$ we use the admissible weight $(x,y)\mapsto\Si(x,y):=\si(x)\si(y)$\,. The choice $\tilde\Si$ in~\cite[Sect.\,7]{DG} is different, but it is easy to check that
$$
\frac{1}{c}\tilde\Si(x,y)\le\Si(x,y)\le\tilde\Si(x,y)^2,\quad\forall\,(x,y)\in\G\times\G
$$
(the first estimate appears at~\cite[(7.1)]{DG} and the second one is obvious), so the seminorms in $\mathcal S_\Si(\G\times\G)$ and $\mathcal S_{\tilde\Si}(\G\times\G)$ are equivalent, the two spaces being the same, with the same topology. Having this in mind, Corollary 7.4 of~\cite{DG}, a version of Schwartz's Kernel Theorem for our spaces, says that the lower horizontal arrow is a topological isomorphism extending~\eqref{intket}.

\smallskip
The upper horizontal arrow is, by definition, the dual (adjoint) of
\begin{equation*}\label{predu}
\mathcal S_\Si(\G\times\G)\cong\mathcal S_\si(\G)\,\overline\otimes\,\mathcal S_\si(\G)\overset{{\sf id}\overline\otimes{\fscr P}}{\longrightarrow}\S_\si(\G)\overline\otimes\mathscr S_{\!\si}(\wG)=:\mathscr S_{\!\si}(\G\times\wG)\,,
\end{equation*}
where the first congruence is~\cite[Th.\,7.1]{DG} (recall that we can switch from $\Si$ to $\tilde\Si$). We use standard properties of the projective tensor products and the very definition of $\mathscr S_{\!\si}(\wG)$ to justify the fact that the restriction ${\sf id}\overline\otimes{\fscr P}$ is an isomorphism.

\smallskip
Finally, to justify the right vertical arrow, by duality, one only has to check that 
\begin{equation*}\label{checknow}
{\sf C}:\S_\Si\big(\G\times\G\big)\to\S_\Si\big(\G\times\G\big)\,,\quad[{\sf C}(K)](x,y):=\Delta(y)^{-1/2}K(x,xy^{-1})
\end{equation*}
is a homeomorphism. This is rather long, but straightforward, using the admissibility of the weight $\si$.
\end{proof}

\begin{Remark}\label{similarius}
{\rm Similar ideas may be put to work to show that $\,{\sf Op}:L^2(\G)\otimes\mathscr B^2(\wG)\rightarrow\mathbb B^2\big[L^2(\G)\big]$ 
restricts to a topological isomorphism $\,{\sf Op}:\mathscr S_{\!\si}\big(\G\times\wG\big)\rightarrow\mathbb L\big[\mathcal S'_\si(\G),\mathcal S_\si(\G)\big]$\,. Setting $\mathbb L(\mathcal M):=\mathbb L(\mathcal M,\mathcal M)$\,,  one also has algebras of global pseudo-differential operators $\mathbb L\big[\mathcal S_\si(\G)\big]$ and $\mathbb L\big[\mathcal S_\si'(\G)\big]$ and a $^*$-algebra of operators $\mathbb L\big[\mathcal S_\si(\G)\big]\cap\mathbb L\big[\mathcal S_\si'(\G)\big]$\,. In general they are not comparable with $\mathbb B^2\big[L^2(\G)\big]$ or $\mathbb B^2\big[L^2(\G)\big]$\,. Behind these operator algebras there are algebras of symbols with composition and involution given by
\begin{equation*}\label{reguli}
A\#B:={\sf Op}^{-1}\big[{\sf Op}(A){\sf Op}(B)\big]\quad{\rm and}\quad A^\#\!:={\sf Op}^{-1}\big[{\sf Op}(A)^*\big]\,.
\end{equation*}
All these may be recast in the framework of Fr\'echet-Hilbert algebras and their associated Moyal algebras~\cite{MP}, but we shall not do this here.
}
\end{Remark}

\section{Kirillov theory and coadjoint orbit quantizations for exponential groups}\label{intkid}

Let us first mention briefly some notions belonging to Kirillov theory~\cite{Ki}. Tacitly, the Lie group will always be type I, connected, simply connected and $N$-dimensional, with Lie algebra $\g$ and dual $\g^\sharp$ of $\g$\,. For the duality, one uses the notation $\<X,\th\>:=\th(X)$\,,\,$(X,\th)\in\g\times\g^\sharp$.
We recall {\it the adjoint and coadjoint actions} of $\G$ on $\mathfrak g$ and $\mathfrak g^\sharp$, respectively, given by
\begin{align*}
{\rm Ad}_x (X) &= \frac d {dt}\Big\vert_{t =0} x\,{\rm e}^{tX} x^{-1}=\frac d {dt}\Big\vert_{t =0} {\rm inn}_x\big({\rm e}^{tX}\big)\,,\\
{\rm Ad}_x^\sharp (\theta) &= \theta \circ{\rm Ad}_{x}^{-1},
\end{align*}
where ${\rm inn}_x(y) = xyx^{-1}$ defines an inner automorphism of $\G$\,. Note that $\exp\circ\,{\sf Ad}_x={\rm inn}_x\circ\exp$\,. For every $\th\in\g^\sharp$, the Lie algebra $\g(\th)$ of the isotropy group $\G(\th):=\big\{x\in\G\mid{\sf Ad}^\sharp_x(\th)=\th\big\}$ can be written as $\g(\th)=\{X\in\g\mid\th\circ{\sf ad}_X=0\}$\,.

\smallskip
It is known that the orbits of the action ${\rm Ad}^\sharp$, called {\it coadjoint orbits}, are invariant symplectic manifolds $\big\{\big(\O,\omega^\O\big)\!\mid\!\O\in\g^\sharp\!/\!_\G\big\}$ (being the symplectic leaves of a Poisson manifold structure on  $\g^\sharp$)\,. So there are invariant measures, {\it the Liouville measures} $\omega_\O$ on the orbits $\O$\,; they are unique up to a positive constant. Let us also set $2n$ for the maximal dimension of the coadjoint orbits, so $N=2n+m$ for some $m\in\N$\,.

\begin{Definition}\label{expobat}
{\rm An exponential group}~\cite{FL} is a Lie group for which the exponential map $\exp:\g\to\G$ is a diffeomorphism (with inverse $\log:\G\to\g$)\,.
\end{Definition} 

The exponential property is stable for Lie subalgebras or quotients.
Connected simply connected nilpotent Lie groups are exponential. Many Frobenius groups (those having at least one open coadjoint orbit) are exponential. On the other hand, exponential groups are solvable, connected and simply connected. 

\smallskip
In the exponential case, we denote by $\kappa:\g^\sharp\!/\!_\G\to\wG$ {\it the Kirillov map}~\cite{Ki}, also admitting the notation $\xi_\O:=\kappa(\O)$\,. We write $\xi\to\O_\xi$ for the inverse map. This type of notation can be extended, setting $\omega^\xi$ for the symplectic form on $\O_\xi$ and $\omega_\xi$ for the corresponding canonical invariant measure.

\smallskip
Some {\it basic facts for exponential groups (or Lie algebras)} are (cf.~\cite{FL,LL,DR}):
\begin{enumerate}
\item[(BF1)]
Kirillov's map $\kappa:\g^\sharp/_{\!\G}\to\wG$ is a homeomorphism. The pushforward $\kappa^{-1}(\wm)$ of the Plancherel measure by $\kappa^{-1}$ is equivalent to the pushforward $\tilde\nu:=q(d\th)$ by the quotient map $q:\g^\sharp\to\g^\sharp\!/_{\!\G}$ of the Lebesgue measure.
\item[(BF2)]
Besides Lebesgue measures $d\th$\, one is also interested in measures $\mu$ on $\g^\sharp$ that are ${\sf Ad}^\sharp$-invariant. These are of the form $d\mu_\Psi(\th)=\Psi(\th)d\th$, for some positive $\Delta^{-1}$-semi-invariant Borel function $\Psi:\g^\sharp\to\R_+$\,. Semi-invariance means that one has almost everywhere
\begin{equation}\label{fronticid}
\Psi\big[{\sf Ad}^\sharp_x(\th)\big]=\Delta(x)^{-1}\Psi(\th)\,.
\end{equation}
It is known that such functions exist, and rational choices are possible.
\item[(BF3)]
Such a measure $\mu_\Psi$ being given, there is a unique basic measure $\nu_\Psi$ on $\g^\sharp\!/_{\!\G}$ over which it decomposes into the fiber measures $\omega_\O$ on the orbits, which formally could be written as $\mu_\Psi=\int_{\g^\sharp\!/_{\!\G}}\omega_\O\,d\nu_\Psi(\O)$\,. The precise meaning is that for every continuous compactly supported function $h:\g^\sharp\to\mathbb C$ one has
\begin{equation}\label{enfin}
\int_{\g^\sharp}h(\th)\Psi(\th)d\th=\int_{\g^\sharp\!/_{\!\G}}\!\Big[\int_\O h(\th)d\omega_\O(\th)\Big]d\nu_\Psi(\O)\,.
\end{equation}
In addition, $\nu_\Psi$ is the image through $\kappa^{-1}$ of one of the Plancherel measures $\wm$ on $\wG$\,.
\item[(BF4)]
It is also proven in~\cite{DR} that, for any coadjoint orbit $\O$\,, the set of functions $\Psi:\O\to\R_+$ satisfying the semi-invariance condition~\eqref{fronticid} is in one-to-one correspondence $\Psi\leftrightarrow{\rm D}_{\Psi,\O}\equiv{\rm D}_{\Psi,\xi_\O}$ with the Duflo-Moore operators associated to the class $\xi_\O\in\wG$ (satisfying~\eqref{eq:invrel}). In a certain sense ${\rm D}_{\Psi,\O}\,d\nu_\Psi(\O)$ is canonical and thus does not depend on $\Psi$\,.
\item[(BF5)]
It is possible to correlate the choices in such a way as to rewrite the Plancherel Inversion Formula (at least) for any $v\in\mathcal D(\G)=C^\infty_{\rm c}(\G)$ in the form
\begin{equation}\label{PIT}
v(\e)=\int_{\g^\sharp\!/_{\!\G}}\!\!{\rm Tr}\Big[{\rm D}^{1/2}_{\Psi,\O}\,\pi_\O(v)\,{\rm D}^{1/2}_{\Psi,\O}\Big]d\nu_\Psi(\O)\,.
\end{equation}
Applying a left translation to $v$ leads to the full {\it orbital form of Plancherel's Inversion Formula}
\begin{equation}\label{FPIT}
v(x)=\int_{\g^\sharp\!/_{\!\G}}\!\!{\rm Tr}\Big[{\rm D}^{1/2}_{\Psi,\O}\,\pi_\O(v)\pi_\O(x)^*\,{\rm D}^{1/2}_{\Psi,\O}\Big]d\nu_\Psi(\O)\,.
\end{equation}
\end{enumerate}

We are going now to recast the global pseudo-differential calculus for exponential groups in the setting of Kirillov's theory. Relying on the facts above, one can replace the space $\mathscr B^2\big(\wG\big):=\int_{\wG}^\oplus\mathbb B^2(\H_\xi)d\wm(\xi)$ with its isomorphic version
\begin{equation}\label{minuna}
\mathscr B^2\big(\g^\sharp\!/_{\!\G}\big):=\int_{\g^\sharp\!/_{\!\G}}^\oplus\!\mathbb B^2(\H_\O)d\nu_\Psi(\O)\,.
\end{equation}
The isomorphism denoted by $\mathscr K$, have its roots in the Kirillov map. {\it The orbital Plancherel transformation}
\begin{equation*}\label{coadrel}
\mathscr P_{\!\rm orb}:L^2(\G)\to\mathscr B^2\big(\g^\sharp\!/_{\!\G}\big)\,,\quad\big[\mathscr P_{\!\rm orb}(v)\big](\O):=\pi_\O(v)D_{\Psi,\O}^{1/2}
\end{equation*}
being unitary, with inverse extending~\eqref{FPIT}, we are lead to {\it the orbital form of the global quantization}
\begin{equation*}\label{cogloqu}
{\sf Op}_{\rm orb}:L^2(\G)\otimes\mathscr B^2\big(\g^\sharp\!/_{\!\G}\big)\to\mathbb B^2\big[L^2(\G)\big]\,,\quad{\sf Op}_{\rm orb}:=\Upsilon\circ{\sf C}\circ({\sf id}\otimes\mathscr P_{\!\rm orb})^{-1},
\end{equation*}
obviously equivalent to ${\sf Op}$ by taking into account the unitary map
\begin{equation*}\label{sandomax}
({\sf id}\otimes\mathscr P_{\rm orb})\circ({\sf id}\otimes\mathscr P)^{-1}\!={\sf id}\otimes\mathscr K:L^2(\G)\otimes\mathscr B^2\big(\wG\big)\to L^2(\G)\otimes\mathscr B^2\big(\g^\sharp\!/_{\!\G}\big)\,.
\end{equation*}
In terms of the Schr\"odinger representation of Remark~\ref{ofi}, there is also the direct definition ${\sf Op}_{\rm orb}={\sf Sch}\circ({\sf id}\otimes\mathscr P_{\!\rm orb})^{-1}$. One gets
\begin{equation}\label{oppad}
\big[{\sf Op}_{\rm orb}(B)u\big](x)= \int_\G \int_{\g^\sharp\!/_{\!\G}}\!\mathrm{Tr}_\xi\!\left[B(x,\O)\,{\rm D}_{\Psi,\O}^{1/2}\,{\pi_\O}(yx^{-1}) \right]\!\Delta(y)^{-\frac 1 2}u(y) \,d\nu_\Psi(\O)\,d\m(y)\,.
\end{equation}

\begin{Remark}\label{revinn}
{\rm Let us set
\begin{equation}\label{unalta}
\Pi_{\rm orb}(x):=\int^\oplus_{\g^\sharp\!/_{\!\G}}\!\pi_\O(x)d\nu_\Psi(\O)\in\int^\oplus_{\g^\sharp\!/_{\!\G}}\mathbb B(\H_\O)d\nu_\Psi(\O)\,,\quad\forall\,x\in\G\,,
\end{equation}
and then $\,{\bf ad}_{\Pi_{\rm orb}(x)}:\mathscr B^2\big(\g^\sharp\!/_{\!\G}\big)\to\mathscr B^2(\g^\sharp\!/_{\!\G})$ by 
\begin{equation*}\label{sofistik}
{\bf ad}_{\Pi_{\rm orb}(x)}\Big(\int^\oplus_{\g^\sharp\!/_{\!\G}}\!T(\O)d\nu_\Psi(\O)\Big):=\int^\oplus_{\g^\sharp\!/_{\!\G}}\!\pi_\O(x)T(\O)\pi_\O(x)^*d\nu_\psi(\O)\,.
\end{equation*}
For $B\in\mathscr B^2\big(\G\times\g^\sharp\!/_{\!\G}\big)$ and $z\in\G$ one has {\it the orbital version of the covariance relation}~\eqref{covariation}
\begin{equation*}\label{covariattion}
{\sf Left}_z\circ{\sf Op}_{\rm orb}(B)\circ{\sf Left}_z^* = {\sf Op}_{\rm orb}\Big[\big({\sf Left}_z\!\otimes{\bf ad}_{\Pi_{\rm orb}(z)}\big)B\Big]\,.
\end{equation*}
}
\end{Remark}

\begin{Remark}\label{vizigot}
{\rm The extension results of Section~\ref{ferfelic} may easily be adapted to this setting, just by replacing the abstract Plancherel map with the orbital one.
}
\end{Remark}

\section{Completely solvable Lie groups and the concrete form of the quantization}\label{csoltkid}

One still needs to replace the orbit space with a nicer version and to describe explicitly the measure $\nu_\Psi$ in this new realization. To do this, for exponential groups, one may rely on~\cite{Cu1,Cu2}, where a parametrization of (large parts of) the orbit space is put into evidence, including an explicit form of the measure $\nu_\Psi$\,.
But the corresponding results for the subclass of completely solvable Lie groups, given in~\cite{CP} and applied to Plancherel theory in~\cite{Cu}, are more manageable. For simplicity, we treat the completely solvable case and dedicate Remark~\ref{vinesiea} to some comments on the changes needed for exponential groups.

\smallskip
We recall that a (connected simply connected) Lie group is called {\it completely solvable} if its (real) $N$-dimensional Lie algebra $\g$ satisfies the following equivalent conditions:
(a) the spectrum (eigenvalues) of all the operators ${\sf ad}_X:\g\to\g$ ($X\in\g$) are all real, (b) there is a sequence of ideals $\{0\}=\g_0\subset\g_1\subset\dots\subset\g_{N-1}\subset\g_N=\g$  with $\dim\g_k=k$ for every $k\in\{1,\dots,N\}$\,. We also recall that ${\rm connected\ simply\ connected\ nilpotent}\Rightarrow{\rm completely\ solvable}\Rightarrow{\rm exponential}\Rightarrow{\rm solvable}$\,
and that all the implications are strict.

\smallskip
There is a complete stratification for completely solvable groups (as well as for the exponential ones), but ``the main fine layer'' will be enough for us. The starting point is given by the inclusions
\begin{equation}\label{scazator}
\g^\sharp\supset\g^\sharp_\circ\supset\Omega=\bigsqcup_{\epsilon\in E}\Omega_\epsilon\,.
\end{equation}
We do not describe the open subset $\g^\sharp_\circ$\,, which will not be directly relevant. We mentioned it only to make the connection with Pukansky's rough stratification, valid for every solvable group.
Let us give a description of the ingredients of~\eqref{scazator} and the role they play. More precise results, also containing parametrizations of the coadjoint orbits and computations of the Liouville measures, are in~\cite[Th.\,1.2\,,\,Lemma\,1.3]{Cu}. The constructions depend on a choice of a Jordan-H\"older basis in $\g$\,.

\smallskip
{\it A Zariski-open subset} of the vector space $W$ (applied to $W=\g^\sharp$\,, for instance) is the union of a family  of sets of the form $\{\th\in W\mid P_\alpha(\th)\ne 0\}$\,, where each $P_\alpha:W\to\R$ is a polynomial. So Zariski-open implies open. If it is not void, a Zariski-open subset is dense and has Lebesgue-negligible complement.

\smallskip
The subset $\Omega$ is non-void and Zariski-open (so it is dense in $\g^\sharp$ and has full Lebesgue measure) and ${\sf Ad}^\sharp$-invariant and all the coadjoint orbits contained in $\Omega$ have maximal dimension $2n$ (but this is also true for the larger set $\g^\sharp_\circ$)\,. It possesses an algebraic subset $\Si$\,, homeomorphic to $\Omega/_{\!\G}$\,, which is a (topological) cross-sections for all the orbits in $\Omega$\,. Let us denote the homeomorphism by $h:\Omega/_{\!\G}\to\Si$\,.

\smallskip
Since there is no nice description of $\Si$\,, the set $\Omega$ is further decomposed as a finite union over $\epsilon\in E$ of disjoint open subsets $\Omega_\epsilon$\,, with $\Si_\epsilon:=\Si\cap\Omega_\epsilon$ easier to describe. One has $\#(E)\le 2^n$. For every $\epsilon\in E$ there is a direct sum decomposition $\g^\sharp=V_\epsilon\oplus V^\epsilon$ with $\dim V_\epsilon\ge2n$ and $\dim V_\epsilon\le m:=N-2n$\,. One also finds a rational function 
\begin{equation}\label{rationalone}
p_\epsilon:\Lambda_\epsilon\subset V_\epsilon\to V^\epsilon,\quad{\rm Gr}(p_\epsilon)=\Si_\epsilon
\end{equation}
defined on a Zariski-open subset $\Lambda_\epsilon$\,; the role of its graph is indicated in~\eqref{rationalone}. Then the canonical projection
\begin{equation*}\label{teintereseaza}
{\rm pr}_\epsilon:\g^\sharp=V_\epsilon\oplus V^\epsilon\to V_\epsilon
\end{equation*}
defines by restriction a rational diffeomorphism $\Si_\epsilon\cong\Lambda_\epsilon$\,. So we get a (rational diffeomorphic) parametrization of the useful cross-section $\Si_\epsilon$ (already a homeomorphic version of the slice $\Omega_\epsilon/_{\!\G}$ of the orbit space). Finally, by composing the suitable restrictions, one gets the homeomorphism
${\sf h}_\epsilon:={\rm pr}_\epsilon\circ h_\epsilon:\Omega_\epsilon/_{\!\G}\to\Lambda_\epsilon$\,,

\begin{equation*}\label{nochegrama}
\begin{diagram}
\node{\Omega_\epsilon/_{\!\G}} \arrow{e,t}{{\sf h}_\epsilon}\arrow{s,l}{h_\epsilon}\node{\Lambda_\epsilon\subset V_\epsilon}\\ 
 \node{\Si_\epsilon}\arrow{ne,r}{{\rm pr}_\epsilon}
\end{diagram}
\end{equation*}
serving among others to endow $\Omega_\epsilon/_{\!\G}$ with the structure of a differentiable manifold.

\smallskip
The remaining part of the job is to transport to the Euclidean open sets $\Lambda_\epsilon$\,, with their Lebesgue measures $d\lambda$\,,  as much relevant theory as we can.
We recall from Section~\ref{intkid} that to any rational semi-invariant function $\Psi:\g^\sharp\to\mathbb R$ one uniquely associates:
\begin{itemize}
\item
a family $\big\{{\rm D}_{\Psi,\O}\mid\O\in\g^\sharp/_{\!\G}\big\}$ of Duflo-Moore operators labeled by the coadjoint orbits,
\item
a measure $\nu_{\Psi}$ on the orbit space, bringing its contribution to the decomposition~\eqref{enfin} and to the Plancherel-type inversion formula~\eqref{FPIT}.
\end{itemize} 
The main result of~\cite{Cu1} tells us that, for a certain polynomial $P_\epsilon:\g^\sharp\to\mathbb R$ (proportional to a Pfaffian acting on elements of the Jorden-H\"older basis), the pushforward by ${\sf h}_\epsilon$ of the restriction $\nu_{\Psi,\epsilon}$ to $\Omega_\epsilon/_{\!\G}$ satisfies
\begin{equation}\label{masura}
\big[{\sf h}_\epsilon(\nu_{\Psi,\epsilon})\big](d\lambda)=:d\Gamma_{\Psi,\epsilon}(\lambda)=\gamma_{\Psi,\epsilon}(\lambda)d\lambda\,,\quad{\rm with}\quad \gamma_{\Psi,\epsilon}(\lambda):=\big\vert\Psi\big[{\rm pr}_\epsilon^{-1}(\lambda)\big]\big\vert\,\big\vert P_\epsilon\big[{\rm pr}_\epsilon^{-1}(\lambda)\big]\big\vert\,.
\end{equation}
Of course one has $\gamma_{\Psi,\epsilon}\circ{\rm pr}_\epsilon=\big\vert\Psi P_\epsilon\big\vert$ (absolute value of a rational function), which is useful for expressing the ``shorter'' pushforward $h_\epsilon(\nu_{\Psi,\epsilon})$\,, but then using ${\rm pr}_\epsilon^{-1}(d\lambda)$ is less attractive.

\begin{Remark}\label{scuarpatrate}
{\rm If $\G$ has irreducible representations that are square integrable, than each refined layer $\Omega_\epsilon$ is a single coadjoint orbit, so $\Lambda_\epsilon$ is a (the) point in a $0$-dimensional vector space $V_\epsilon$ and it is of course an atom.
}
\end{Remark}

So now, specifying a ``generic'' orbit $\O\in\Omega/_{\!\G}$ reduces to specifying some $\epsilon\in E$ and some element $\lambda\in\Lambda_\epsilon\subset V_\epsilon$\,. As a consequence of the above,~\eqref{PIT} reads now
\begin{equation*}\label{PITC}
v(\e)=\sum_{\epsilon\in E}\,\int_{\Lambda_\epsilon}\!{\rm Tr}\Big[{\rm D}^{1/2}_{\Psi,\lambda,\epsilon}\,\pi_{\lambda,\epsilon}(v)\,{\rm D}^{1/2}_{\Psi,\lambda,\epsilon}\Big]\big\vert\Psi\big[{\rm pr}_\epsilon^{-1}(\lambda)\big]\big\vert\,\big\vert P_\epsilon\big[{\rm pr}_\epsilon^{-1}(\lambda)\big]\big\vert\,d\lambda\,,
\end{equation*}
only involving integration on a finite number of open subsets of the vector spaces $V_\epsilon$\,. By a left translation we arrive at {\it the concrete Plancherel inversion formula}
\begin{equation}\label{PITCo}
v(x)=\sum_{\epsilon\in E}\,\int_{\Lambda_\epsilon}\!{\rm Tr}\Big[{\rm D}^{1/2}_{\Psi,\lambda,\epsilon}\,\pi_{\lambda,\epsilon}(v)\pi_{\lambda,\epsilon}(x)^*\,{\rm D}^{1/2}_{\Psi,\lambda,\epsilon}\Big]\big\vert\Psi\big[{\rm pr}_\epsilon^{-1}(\lambda)\big]\big\vert\,\big\vert P_\epsilon\big[{\rm pr}_\epsilon^{-1}(\lambda)\big]\big\vert\,d\lambda\,.
\end{equation}

\begin{Corollary}\label{ladracu}
Let us also consider the open subsets $\big\{\wG_\epsilon:=\kappa\big(\Omega_\epsilon/_{\!\G}\big)\mid\epsilon\in E\big\}$ of the unitary dual $\wG$\,. Then $\wG_\Omega:=\bigsqcup_{\epsilon\in E}\wG_\epsilon$ is an (open) dense subset of $\,\wG$ with Plancherel-negligible complement. In addition each $\wG_\epsilon\cong\Omega_\epsilon/_{\!\G}\cong\Lambda_\epsilon$ is endowed with a structure of a differentiable manifold.
\end{Corollary}

These facts suggest considering the Hilbert space
\begin{equation}\label{hilbspac}
\mathscr B^2_{\rm con}(\Lambda):=\bigoplus_{\epsilon\in E}\,\int^\oplus_{\Lambda_\epsilon}\mathbb B^2\big(\H_{\O_\lambda}\big)d\Gamma_{\Psi,\epsilon}(\lambda)\cong\bigoplus_{\epsilon\in E}\,\int^\oplus_{\Omega_\epsilon/_{\!\G}}\!\mathbb B^2\big(\H_{\O}\big)d\nu_{\Psi,\epsilon}(\O)\,,
\end{equation}
isomorphic to 
\begin{equation*}\label{crichet}
\mathscr B^2\big(\wG\big):=\int_{\wG}^\oplus\mathbb B^2(\H_\xi)d\wm(\xi)\cong\bigoplus_{\epsilon\in E}\,\int^\oplus_{\wG_\epsilon}\!\mathbb B^2\big(\H_{\xi}\big)d\wm_{\epsilon}(\xi)\,,
\end{equation*}
where $\wm_\epsilon$ is the restriction of the Plancherel measure to the open subset $\wG_\epsilon:=\kappa\big(\Omega_\epsilon/_{\!\G}\big)$\,. 

\smallskip
{\it The concrete Plancherel transformation}
\begin{equation*}\label{coadrell}
\mathscr P_{\!\rm con}:L^2(\G)\to\mathscr B^2_{\rm con}\big(\Lambda)\,,\quad\big[\mathscr P_{\!\rm con}(v)\big](\lambda):=\pi_{\O_\lambda}(v){\rm D}_{\Psi,\O_\lambda}^{1/2}\equiv\pi_{\lambda}(v){\rm D}_{\Psi,\lambda}^{1/2}
\end{equation*}
being unitary, with inverse indicated by~\eqref{PITCo}, we are lead to {\it the concrete form of the global quantization}
\begin{equation*}\label{cogloqur}
{\sf Op}_{\rm con}:L^2(\G)\otimes\mathscr B_{\rm con}^2(\Lambda)\to\mathbb B^2\big[L^2(\G)\big]\,,\quad{\sf Op}_{\rm con}:=\Upsilon\circ{\sf C}\circ({\sf id}\otimes\mathscr P_{\!\rm con})^{-1},
\end{equation*}
obviously equivalent to~\eqref{finop} by taking into account the unitary map
\begin{equation*}\label{sandomarx}
({\sf id}\otimes\mathscr P_{\rm con})\circ({\sf id}\otimes\mathscr P)^{-1}:L^2(\G)\otimes\mathscr B^2\big(\wG\big)\to L^2(\G)\otimes\mathscr B^2_{\rm con}\big(\Lambda)\,.
\end{equation*}
In terms of the Schr\"odinger representation there is also the direct definition
\begin{equation*}\label{cogloqusch}
{\sf Op}_{\rm con}={\sf Sch}\circ({\sf id}\otimes\mathscr P_{\!\rm con})^{-1}.
\end{equation*}
One gets
\begin{equation}\label{oppadd}
\big[{\sf Op}_{\rm con}(\mathcal B)u\big](x)=\sum_{\epsilon\in E}\,\int_\G \int_{\Lambda_\epsilon}\!\mathrm{Tr}_\xi\!\left[\mathcal B(x,\lambda)\,{\rm D}_{\Psi,\lambda}^{1/2}\,{\pi_\lambda}(yx^{-1}) \right]\!\Delta(y)^{-\frac 1 2}u(y) \,\gamma_{\Psi,\epsilon}(\lambda)d\lambda\,d\m(y)\,.
\end{equation}

The extension results of Section~\ref{ferfelic} may easily be adapted to this setting, by replacing the abstract Plancherel map with the concrete one.

\begin{Remark}\label{rewin}
{\rm There is an obvious adaptation of the covariance results Proposition~\ref{avenir} or Remark~\ref{revinn} to the concrete framework. Setting
\begin{equation*}\label{unaltai}
\Pi_{\rm con}(x):=\bigoplus_{\epsilon\in E}\,\int^\oplus_{\Lambda_\epsilon}\!\pi_\lambda(x)\gamma_{\Psi,\epsilon}(\lambda)d\lambda\in\bigoplus_{\epsilon\in E}\,\int^\oplus_{\Lambda_\epsilon}\mathbb B(\H_\lambda)\gamma_{\Psi,\epsilon}(\lambda)d\lambda\,,\quad\forall\,x\in\G\,,
\end{equation*}
in terms of the corresponding $\,{\bf ad}_{\Pi_{\rm con}(x)}:\mathscr B^2_{\rm con}(\Lambda)\to\mathscr B^2_{\rm con}(\Lambda)$\,, one has {\it the concrete covariance relation} for $\mathcal B\in\mathscr B^2(\G)\otimes\mathscr B^2_{\rm con}(\Lambda)\equiv\mathscr B^2(\G\times\Lambda)$ and $z\in\G$  
\begin{equation*}\label{covariadtion}
{\sf Left}_z\circ{\sf Op}_{\rm con}(\mathcal B)\circ{\sf Left}_z^* = {\sf Op}_{\rm con}\Big[\big({\sf Left}_z\!\otimes{\bf ad}_{\Pi_{\rm con}(z)}\big)\mathcal B\Big]\,.
\end{equation*}
}
\end{Remark}

\begin{Example}\label{liucil}
{\rm One of the simplest situations occurs when $\G$ is a connected simply connected nilpotent Lie group for which there are flat coadjoint orbits~\cite{CG,MW} (equivalently, if there are square integrable modulo the center irreducible representations). Of course, such groups are unimodular, so already the modular function, the Duflo-Moore operators and the function $\Psi$ are trivial. In this case, one can take $E$ to be a singleton and $\Lambda$ is a Zariski open subset of the dual $\mathfrak z^\sharp$ of the center $\mathfrak z$ of the Lie algebra $\g$\,. Specifically,
$$
\Lambda=\big\{\th\in\mathfrak z^\sharp\mid {\rm Pf}(\th)\ne 0\big\}\,,
$$
where ${\rm Pf}:\mathfrak z\to\R$ is a certain explicit Pfaffian, a coadjoint-invariant homogeneous polynomial. The density of the measure $\Gamma$ on $\Lambda\subset\mathfrak g^\sharp$ with respect to the Lebesgue measure of the vector space $\mathfrak g^\sharp$ is proportional to $|{\rm Pf}(\cdot)|$\,. This case has been considered from several points of view in~\cite{MR1}, so we are not going to give further details.
}
\end{Example}

\begin{Remark}\label{vinesiea}
{\rm We describe briefly some complications occurring for exponential groups that are not completely solvable; the details can be found in~\cite{Cu1,Cu2}:
\begin{itemize}
\item
This time one must start with a basis of the complexification $\g_{\mathbb C}$ of the Lie algebra $\g$\,, satisfying some precise conditions.
\item
A convenient big open invariant subset $\Omega\subset\g^\sharp$ is a bundle over a cross-section $\Si$\,. 
\item
This cross-section $\Si$\,, still an algebraic manifold, has a more complicated nature. Even after suitable localizations, it is not modeled by Zariski-open subsets of some vector subspace $V$. Instead of $V$ one uses Cartesian products of pieces that are of one of the forms $\R,\mathbb C,\{-1,1\}$ or $\T$\,. The natural ``Lebesgue'' measures of these pieces contribute now to the part $d\lambda$ in the concrete measure~\eqref{masura}.
\item
The density in~\eqref{masura} needs a correction by a factor involving some of the roots of the coadjoint action.
\item
Some functions needed for parametrization, instead of rational, are now only real analytic.
\item
The Duflo-Moore operators take a precise form (multiplication by the modular function), if the Vergne polarization is chosen for the realizations of the irreducible representations. 
\end{itemize}
}
\end{Remark}

\section{The connected simply connected solvable Bianchi groups}\label{cracit}

Although the results of the previous section are explicit (if one follows closely the details in~\cite{Cu,Cu1}), they are quite complicated. In certain particular cases it might be easier to rely on some direct treatment, as~\cite{AV}.

\smallskip
We recall that the $3$-dimensional real Lie algebras have been long ago classified modulo isomorphy by Bianchi in 9 classes $\g_{\rm W}$ with ${\rm W=I,I\!I,\dots,I\!X}$\,. The classes $\g_{\rm V\!I\!I\!I}$ and $\g_{\rm I\!X}$ are simple and they do not concern us here. The remaining ones are all solvable, most of them exponential actually (see below). The corresponding connected simply connected solvable Lie groups $\big\{\G_{\rm W}\mid {\rm W=I,I\!I,\dots,V\!I\!I}\big\}$ are the topic of this section. The classes are indicated specifying a special member easy to identify.

\smallskip
The first Bianchi group is the vector space $\G_{\rm I}=\R^3$, and it is clear that the global quantization in this case is just the usual Kohn-Nirenberg quantization~\cite{Fo,Sh}, the initial motivation of all our approach. We do not review it here.

\smallskip
The second one is the Heisenberg group $\G_{\rm I\!I}=\mathbb H_3$ (nilpotent and thus unimodular). Its global quantization has been repetedly considered before, so we will not treat it. See~\cite[Sect.\,6]{FR} for a very detailed approach (including many references) as well as Example~\ref{liucil}.

\smallskip
The third one is $\G_{\rm I\!I\!I}={\sf Aff}\times\R$\,, the direct product between the affine group ${\sf Aff}$ (also called ``$ax+b$'') of the real line and the one-dimensional vector space $\R$\,. We treated ${\sf Aff}$ in~\cite{MS}, see also~\cite{Fan,GGBV}, and for the direct product one can use Remark~\ref{likely}. For completeness, we are now going to indicate briefly the form of the global pseudo-differential operators.

\smallskip
Thus $\G_{\rm I\!I\!I}={\sf Aff}\times\R=\R_+\times\R\times\R$ with composition law
$$
(a,b;c)\cdot\big(a',b';c'\big) = \big(aa',ab'+b;c+c'\big)\,.
$$
The left Haar measure is $\left\vert a \right\vert^{\scriptscriptstyle -2}\!dadbdc$\,, and the right Haar measure is $\left\vert a\right\vert^{-1}\!dadbdc$\,, hence the modular function is given by
$\Delta(a,b;c) = \left\vert a \right\vert^{\scriptscriptstyle -1}\!.$ 

The unitary dual $\widehat{\sf Aff}$ consists only of two points with strictly positive Plancherel measure and a null set of one-dimensional representations~\cite[Sect. 6.7]{Fo}, that we can neglect. The two main representations are square integrable and act on $\H_\pm=L^2(\mathbb R_\pm)$ by
$$
[\pi_{\pm}(a,b)\varphi](s) = \left\vert a \right\vert^{1/2}\!e^{2\pi i bs} \varphi(a s)\,.
$$
The Duflo-Moore operator are $({\rm D}_\pm\varphi)(s) = \left\vert s \right\vert\varphi(s)\,$. Correspondingly, the dual $\mathfrak{aff}^\sharp$ of the Lie algebra, identified with $\R^2$, has the two open coadjoint orbits $\O_{\pm}:=\{(\alpha,\beta)\in\R^2\mid \pm\alpha>0\}$ (half planes), while all $(\alpha,0)$ are fixed points. Remark~\ref{scuarpatrate} is relevant here. The contribution  of ${\sf Aff}$ to the space of symbols
is $L^2({\sf Aff})\otimes\big[\mathbb B^2(\H_-)\oplus\mathbb B^2(\H_+)\big]$ (see~\eqref{hilbspac}; here $E=\{\pm\}$).

\smallskip
Taking into consideration the dual $\widehat\R$ of $\R$\,, previous formulae and Remark~\ref{likely}, the quantization reads
$$
\begin{aligned}
&[{\sf Op}(\mathcal B)u](a,b;c)\\
&=\sum_{\pm}\int_0^\infty\!\!\!\!\iiint_{\mathbb R^3} \frac{e^{-ic\zeta}}{\left\vert a'\right\vert^{3/2} }\,\mathrm{Tr}_\pm\!\left(\mathcal B_\pm(a,b;c\!\mid\!\zeta){\rm D}_\pm^{\frac 1 2}\pi_\pm\!\left( b - \frac a {a'} b',\frac a {a'}\right)\!^*\right) u(a',b';c') \, da'db'dc'd\zeta\\
&=\sum_{\pm}\int_0^\infty\!\!\!\!\iint_{\mathbb R^2} \mathrm{Tr}_\pm\!\left(\tilde{\mathcal B}_\pm(a,b;c\!\mid\!c){\rm D}_\pm^{\frac 1 2}\pi_\pm\!\left( b - \frac a {a'} b',\frac a {a'}\right)\!^*\right) u(a',b';c') \, \frac{da'db'dc'}{\left\vert a'\right\vert^{3/2}}\,.
\end{aligned}
$$
By $\ \tilde{}\ $ we indicated a (suitably normalized) partial Euclidean Fourier transformation. Traces are computed with respect to the Hilbert spaces $\H_{\pm}$ and the symbol $\,\mathcal B\equiv(\mathcal B _-,\mathcal B_+)$ can be seen as an element of 
$$
\Big(L^2({\sf Aff})\otimes\big[\mathbb B^2(\H_-)\oplus\mathbb B^2(\H_+)\big]\Big)\otimes\Big(L^2(\R)\otimes L^2(\widehat\R)\Big)\cong L^2\big(\G\times\widehat\R;\mathbb B^2(\H_-)\big)\oplus L^2\big(\G\times\widehat\R;\mathbb B^2(\H_+)\big)\,.
$$ 

We treat now the {\it three-dimensional Lie algebras with two-dimensional derived algebra}.
Let us denote by $\g'$ the vector space of the Lie algebra $\g$ generated by commutators (actually it is an ideal). One has 
$$
\dim\g'_{\rm I}=0\,,\quad\dim\g'_{\rm I\!I}=\dim\g'_{\rm I\!I\!I}=1\,,\quad\dim\g'_{\rm V\!I\!I\!I}=\dim\g'_{\rm I\!X}=3\,.
$$
For the remaining Bianchi Lie algebras $\g_{\rm W}$ one has $\dim\g'_{\rm W}=2$ and this leads to a common structure. Each one is (isomorphic with) the semi-direct product $\g_{\rm W}=\R^2\!\rtimes_{M_{\rm W}}\!\R$ of a two-dimensional Abelian Lie algebra by $\R$\,, defined by a linear automorphism $M_{\rm W}:\R^2\to\R^2$. Let us skip the index ${\rm W}$ for a while and use notations as $\big((\alpha,\beta);\gamma\big)\equiv(\alpha,\beta;\gamma)$ for elements of $\g$\,. Then the Lie bracket is
\begin{equation}\label{taraboanta}
[(\alpha,\beta;\gamma),(\alpha',\beta';\gamma')]=\big(\gamma M(\alpha',\beta')-\gamma' M(\alpha,\beta);0\big)\,.
\end{equation}
Consequently, the corresponding Lie groups have the form $\G=\R^2\!\rtimes_{\mathbf M}\!\R$\,, which is a semi-direct product defined by the action of $\R$ by automorphisms of the vector group $\R^2$ given by
\begin{equation}
\mathbf M:\R\to{\sf Aut}(\R^2)\,,\quad \mathbf M_c:=e^{cM}.
\end{equation}
The composition law is
\begin{equation*}\label{mydoor}
(a,b;c)(a',b';c')=\big((a,b)+\mathbf M_c(a',b');c+c'\big)\,.
\end{equation*}
Then a left Haar measure of $\G$ is $d\m(a,b;c)=\det\!\big(e^{-cM}\big) dadbdc$ and the modular function is $\Delta(a,b;c)=\det\!\big(e^{-cM}\big)$ (thus the Lebesgue measure $dadbdc$ is a right Haar measure).

\smallskip
Currey's parametrization (only depending on a choice of a suitable Jordan-H\"older basis) is not so complicated in three dimensions. But we found convenient to rely on results from~\cite{AV}, in which the authors found easier to apply directly Mackey's induced representation theory for semidirect products than the parametrization methods of Currey. Different (but equivalent) descriptions rely on different choices for topological cross-sections in the orbit space. The emphasis in~\cite{AV} is on the orbit structure of the contragredient action $\mathbf M^\perp$ of $\R$ on the dual $\widehat{\R^2}\equiv\R^2$, {\it and this will allow below a unified treatment}.  But in cases ${\rm V\!I}$ and ${\rm V\!I\!I}$ Currey's cross-sections could be considered simpler.

\smallskip
The main conclusion is that the unitary dual $\wG$ can be identified with the disjoint union of a family of fixed points, labeled by $\R$\,, and  not contributing to the Plancherel measure, and a  quotient $\R^2_\bu/\mathbf M^\perp$\,, where $\R^2_\bu:=\R^2\setminus\{0\}$\,. The generic classes of irreducible representations $\xi$ are actually labeled by a parameter $\sigma$ belonging to a cross-section $\Si\subset\R^2_\bu$\,. The ($1$-dimensional) algebraic submanifolds $\Si$\,, composed of one or several connected components, and its relevant measures $d\rho(\si)=\rho(\si)d\si$, will be indicated below, via a parametrization, case by case. The generic classes of irreducible representations can all be realized on the Hilbert space $\H_\si=L^2(\R)$ as
\begin{equation*}\label{realized}
\big[\pi_\sigma(a,b;c)\varphi\big](t):=e^{i\sigma\cdot e^{-tM}\!(a,b)}\varphi(t-c)\,,
\end{equation*}
while the Duflo-Moore operators are multiplication operators by the modular function $\Delta$ (only depending on the last variable $c\in\R$)\,. 
Making use of this limited information, one writes the global quantization as
\begin{equation}\label{operatoruku}
\begin{aligned}
\ [{\sf Op}(\mathcal B)u](a,b;c)=\int_{\R^3}\!\int_{\Si} \ &\mathrm{Tr}\left(\mathcal B(a,b;c\!\mid\!\si){\rm D}_\si^{\frac 1 2}{\pi_\si}\big((a',b')-e^{(c'-c)M}(a,b);c'-c\big) \right)\\
&u(a',b';c') \det\!\big(e^{-c'M}\big)^{1/2}da'db'dc'\rho(\si)d\si.
\end{aligned}
\end{equation}

To be more specific, one has to invoke the Bianchi classification~\cite{EW,Ja} and results describing the parametrization spaces from~\cite{AV}.
It can be shown that two such semi-direct product Lie algebras given respectively by $M_1$ and $M_2$ are isomorphic  if and only if the endomorphisms $M_1$ and $M_2$ are similar up to a scaling. Combining this with the real form of Jordan's canonical decomposition, one gets the cases
\begin{equation*}\label{matrixM}
M_{\rm I\!V}=
\begin{bmatrix}
 1       & 0 \\
 1      & 1
\end{bmatrix}
,\ 
M_{\rm V}=
\begin{bmatrix}
1  & 0 \\
0 & 1 
\end{bmatrix}
,\ 
M^{(q)}_{\rm V\!I}=
\begin{bmatrix}
1  & 0 \\
0 & -q 
\end{bmatrix}
,\ 
M^{(p)}_{\rm V\!I\!I}=
\begin{bmatrix}
p  & -1 \\
1 & p 
\end{bmatrix}
,
\end{equation*}
where $q\ne 0,-1$ and $p\ge 0$\,, with corresponding group actions
\begin{equation*}\label{matrixbfM}
e^{tM_{\rm I\!V}}\!=\!
\begin{bmatrix}
 e^t       & 0 \\
 te^t      & e^t
\end{bmatrix}
,\ 
e^{tM_{\rm V}}\!=\!
\begin{bmatrix}
e^t  & 0 \\
0 & e^t 
\end{bmatrix}
,\ 
e^{tM^{(q)}_{\rm V\!I}}\!=\!
\begin{bmatrix}
e^t  & 0 \\
0 & e^{-qt} 
\end{bmatrix}
,\ 
e^{t M^{(p)}_{\rm V\!I\!I}}\!=\!
\begin{bmatrix}
e^{pt}\cos t  & -e^{pt}\sin t \\
e^{pt}\sin t & e^{pt}\cos t 
\end{bmatrix}
.
\end{equation*}
All the groups are completely solvable, with the exception of $\G^{(p)}_{\rm V\!I\!I}$ (for $p>0$ called {\it the Gr\'elaud group} in~\cite{FL}). In fact $\G^{(0)}_{\rm V\!I\!I}$ is not even exponential. The single unimodular ones correspond to $M^{(1)}_{\rm V\!I}$ and $M^{(0)}_{\rm V\!I\!I}$\,. In each case there is a parametrization $\Lambda$ of $\Si$ and a measure $d\Gamma(\lambda)=\gamma(\lambda)d\lambda$\,, that we now recall. By $d\lambda$ we denote the Lebesgue measure on the indicated segments or, in one case, on the circle. The densities $\gamma$ are indicated up to a strictly positive constant.
\begin{enumerate}
\item[({\rm I\!V})]
$\Lambda=(-\infty,0)\sqcup(0,\infty)$\,, \,$\gamma(\lambda)=1+|\lambda|$\,,
\item[({\rm V})]
$\Lambda=\T:=\R/\Z$\,, \,$\gamma(\lambda)=1$\,,
\item[({\rm V\!I})]
$\Lambda=(0,\infty)\times\{1,2,3,4\}$\,, \,$\gamma_q(\lambda)=q^{\lambda({\rm mod} 2)}$\,,
\item[({\rm V\!I\!I})]
$\Lambda=(e^{-p\pi},1]\sqcup[1,e^{p\pi})$\,, \,$\gamma(\lambda)=|\lambda|$\ \,(here $p>0$)\,.
\end{enumerate}
Details, including more formulae and a picture of the orbits and of the cross-sections $\Si$\,, can be found in~\cite{AV}; they help to understand the splittings. The results are roughly compatible with Currey's theory. We recall however that both Currey's and the parametrizations from~\cite{AV} depend on choices and can be subject to modifications.

\smallskip
The parametrizations and the measures should be inserted into~\eqref{operatoruku} to replace the cross-section $\Si$ and the measure $\rho(\si)d\si$. For example, for the case ({\rm V}) one gets
\begin{equation}\label{operatoraku}
\begin{aligned}
\ [{\sf Op}(\mathcal B)u](a,b;c)=\int_{\R^3}\!\int_{\T} \ &\mathrm{Tr}\left(\mathcal B(a,b;c\!\mid\!\si){\rm D}_\si^{\frac 1 2}{\pi_\si}\big((a',b')-e^{(c'-c)M}(a,b);c'-c\big) \right)\\
&u(a',b';c') e^{-c'}da'db'dc'd\lambda\,.
\end{aligned}
\end{equation}

{\bf Acknowledgments.} M. Sandoval has been supported by CONICYT-PCHA/Mag\'{\i}sterNacional/2016-22160383 and partially by N\'ucleo Milenio de F\'{\i}sica Matem\'atica RC120002. M. M\u antoiu is supported by the Fondecyt Project 1160359. The authors are grateful to two anonymous referees for useful comments that contributed to improve the final form of the manuscript.

\bigskip

Marius M\u antoiu, Maximiliano Sandoval:
\endgraf
Departamento de Matem\'aticas, Universidad de Chile
\endgraf
Casilla 653, Las Palmeras 3425, Ñuñoa, Santiago, Chile
\endgraf
{\it E-mail addresses:} {\rm mantoiu@uchile.cl} {\rm msandova@protonmail.com}


\begin{thebibliography}{1000}

\bibitem{AFK} S.\,T. Ali, H. F\"uhr and A. Krasowska: \textit{Plancherel Inversion as Unified Approach to Wavelet Transforms and Wigner Functions}, Ann. Henri Poincar\'e \textbf{4}, 1015--1050 (2003).

\bibitem{AV} Z. Avetisyan and R. Verch: \textit{Explicit Harmonic and Spectral Analysis in Bianchi I-VII-Type Cosmologies}, Classical and Quantum Gravity, \textbf{30}(15), (2013).

\bibitem{BFKG} H. Bahouri., C. Fermanian-Kammerer and I. Gallagher: \textit{Phase Space Analysis and Pseudodifferential Calculus on the Heisenberg Group}, Ast\'erisque, \textbf{342}, (2012).

\bibitem{Br} F. Bruhat: \textit{Distributions sur un groupe localement compact et applications a l'\'etude des repr\'esentations des groupes $p$-adiques}, 
Bull. Soc. Math. France \textbf{89}, 43--75, (1961).

\bibitem{CGGP} M. Christ, D. Geller, P. G\l owacki and L. Polin: {\it Pseudodifferential operators on groups with dilations}, Duke Math. J., {\bf 68}, no. 1, 31--65, (1992). 

\bibitem{CG} L.\,J. Corwin and F. P. Greenleaf: \emph{Representations of Nilpotent Lie Groups and their Applications}, Cambridge University Press, 1990.

\bibitem{Cu} B.\,N. Currey: \textit{An Explicit Plancherel Formula for Completely Solvable Lie Groups}, Michigan Math. J. \textbf{38}(1), 75--87, (1991).

\bibitem{Cu1} B.\,N. Currey: \textit{The Structure of the Space of Coadjoint Orbits of an Exponential Solvable Lie Group}, Trans. Amer. Math. Soc. {\bf 332}(1), 241--269, (1992).

\bibitem{Cu2} B.\,N. Currey: \textit{Explicit Orbital Parameters and the Plancherel Measure for Exponential Lie Groups}, Pacific J. Math., {\bf 219}(1), 97--138, (2005).

\bibitem{CP} B.\,N. Currey and R. C. Penney: \textit{The Structure of the Space of Co-adjoint Orbits of a Completely Solvable Lie Group}, Michigan Math. J. \textbf{36}(2), 309--320, (1989).

\bibitem{DG} E. David-Guillou: \textit{Schwartz Functions, Tempered Distributions, and Kernel Theorem on Solvable Lie Groups}, Preprint arXiv and to appear in Ann. Inst. Fourier.

\bibitem{DM} M. Duflo and C.\,C. Moore: \emph{On the Regular Representation of a Nonunimodular Locally Compact Group}, J. Funct. Anal. {\bf 21}, 209--243 (1976).

\bibitem{DR} M. Duflo and M. Ra\"{\i}s: \textit{Sur l'analyse harmonique sur les groupes de Lie r\'esolubles}, Ann. Scient. \'Ec. Norm. Sup., \textbf{9}, 107--144, (1976).

\bibitem{Di} J. Dixmier: \emph{Les $C^*$-alg\`ebres et leurs repr\'esentations}, Cahiers scientifiques, XXIX, Gauthier-Villars \' Editeurs, Paris, 1969.

\bibitem{EW} K. Erdmann and M. J. Wildon: \textit{Introduction to Lie Algebras}, Springer, London, 2007.

\bibitem{Fan} Q. Fan: \textit{Symbol Calculus on the Affine group ``ax+b''}, Studia Math. \textbf{115}(3), 207--217, (1995).

\bibitem{FR} V. Fischer and M.  Ruzhansky: \emph{Quantization on Nilpotent Lie Groups}, Progress in Mathematics, Birkh\"auser, vol. \textbf{314}, 2016. 

\bibitem{FR1} V. Fischer and M. Ruzhansky: \emph{A Pseudo-differential Calculus on the Heisenberg Group}, C. R. Acad. Sci. Paris, Ser I, {\bf 352}, 197--204,  (2014).

\bibitem{Fu} H. F\"uhr: \emph{Abstract Harmonic Analysis of Continuous Wavelet Transforms}, Lecture Notes in Mathematics {\bf 1863}, Springer-Verlag Berlin Heidelberg, 2005.

\bibitem{Fo} G.\,B. Folland: \emph{Harmonic Analysis in Phase Space}, Annals of Mathematics Studies, {\bf 122}. Princeton University Press, Princeton, NJ, 1989.

\bibitem{Fo1} G.\,B. Folland: \emph{A Course in Abstract Harmonic Analysis}, CRC Press, Boca Raton Ann Arbor London Tokio, 1995. 

\bibitem{FL}  H. Fujiwara and  J. Ludwig:  \textit{Harmonic Analysis on Exponential Solvable Lie Groups}, Springer Monographs in Mathematics, Springer, Tokyo, 2015. 

\bibitem{GGBV} V. Gayral, J.\,M. Gracia-Bond\'ia, and J.\,C.V\'arilly: \textit{Fourier Analysis on the Affine Group, Quantization and Noncompact Connes Geometries}, J. Noncommut. Geom. \textbf{2}, 215--261, (2008).

\bibitem{Glo} P. G\l owacki, {\it The Melin Calculus for General Homogeneous Groups,} Ark. Mat., {\bf 45}, no. 1, 31--48, (2007). 

\bibitem{Ja} N. Jacobson: \textit{Lie Algebras}, Dover Publications, New York, 1979

\bibitem{Ki} A.\,A. Kirillov: \textit{Lectures on the Orbit Method}, Graduate Studies in Mathematics, \textbf{64}, American Mathematical Society, Providence, RI, 2004. 

\bibitem{LL} H. Leptin and J. Ludwig: \textit{Unitary Representation Theory of Exponential Lie Groups}, De Gruyter Expositions in Mathematics \textbf{18}, Berlin New York, 1994.

\bibitem{Ma} M. M\u antoiu: \textit{Essential Spectrum and Fredholm Properties for Operators on Locally Compact Groups}, J. Oper. Th. \textbf{77}, no. 2, 481--501, (2017).

\bibitem{Ma1} M. M\u antoiu: \textit{A Positive Quantization on Type $I$ Locally Compact Groups}, Math. Nachrichten, \textbf{292}(5), 1043--1055, (2019).

\bibitem{MP} M. M\u antoiu  and R. Purice: \textit{On Fr\'echet-Hilbert Algebras}, Arch. Math. \textbf{103}(2), 157--166,  (2014).

\bibitem{MR} M. M\u antoiu  and M. Ruzhansky:  \textit{Pseudo-Differential Operators, Wigner transform and Weyl Systems on Type I Locally Compact Groups}, Doc. Math. \textbf{32}, 1539--1592, (2017).

\bibitem{MR1} M. M\u antoiu  and M. Ruzhansky:\textit{ Quantizations on Nilpotent Lie Groups and Algebras Having Flat Coadjoint Orbits}, to appear in J. Geometric Analysis, (2019).

\bibitem{MS} M. M\u antoiu and M. Sandoval: \textit{Pseudo-differential Operators Associated to General Type I Locally Compact Groups}, Analysis and Partial Differential Equations: Perspectives from Developing Countries, Conference Proceedings, Imperial College, 172--190, 2016.

\bibitem{Melin} A. Melin: {\it Parametrix Constructions for Right Invariant Differential Operators on Nilpotent Groups}, Ann. Global Anal. Geom., {\bf 1}, no. 1, 79--130, (1983).

\bibitem{MW} C. C. Moore and J. Wolf: \textit{Square Integrable Representations
  of Nilpotent Groups}, Trans.\@ of the AMS, {\bf 85}, 445--462, (1973).

\bibitem{Ng} B-K. Nguyen: \textit{Pseudo-Differential Calculus on Generalized Motion Groups}, PhD Thesis, Imperial College, London.

\bibitem{RT2} M. Ruzhansky and V. Turunen: {\it Quantization of Pseudo-differential Operators on the Torus}, J. Fourier Anal. Appl. {\bf 16}, 943--982, (2010).

\bibitem{RT} M. Ruzhansky and V. Turunen: {\it Pseudo-differential Operators and Symmetries}, Pseudo-Differential Operators: Theory and Applications {\bf 2}, Birkh\"auser Verlag, 2010. 

\bibitem{RT1} M. Ruzhansky and V. Turunen: {\it Global Quantization of Pseudo-differential Operators on Compact Lie Groups, $SU(2)$ and $3$-Sphere}, Int. Math. Res. Not., no. 11, 2439--2496, (2013).

\bibitem{RTW} M. Ruzhansky, V. Turunen and J. Wirth: {\it H\"ormander-Class of Pseudo-Differential Operators on Compact Lie Groups and Global Hypoellipticity}, J. Fourier Anal. Appl., {\bf 20}, 476--499, (2014).

\bibitem{RW} M. Ruzhansky and J. Wirth: \emph{Global Functional Calculus for Operators on Compact Lie Groups}, J. Funct. Anal. {\bf 267}, 144--172,  (2014).

\bibitem{Sch} L.\,B. Schweitzer: \emph{Dense $m$-Convex Fr\'echet Subalgebras of Operator Algebra Crossed Products by Lie Groups}. Internat. J. Math., \textbf{4}(4), 601--673, (1993).

\bibitem{Sh} M. A. Shubin: \emph{Pseudodifferential Operators and Spectral Theory}, Springer Series in Soviet Mathematics, Springer-Verlag, Berlin, 1987.

\bibitem{Tat} N. Tatsuuma: \emph{Plancherel Formula for Non-unimodular Locally Compact Groups}, J. Math. Kyoto Univ. \textbf{12}, 179--261, (1972).

\bibitem{Ta} M. Taylor: \emph{Noncommutative Microlocal Analysis}, Mem. Amer. Math. Soc., {\bf 313}, Providence R.I. 1984.

\bibitem{Ze} S. Zeldich: \emph{Pseudodifferential Analysis on Hyperbolic Surfaces}, J. Funct. Anal. {\bf 68}, 72--105, (1986).

\end{thebibliography}
\end{document}